\documentclass[11pt, a4paper]{amsart}
\usepackage{amsfonts,amssymb,amsmath, amsthm}
\usepackage{mathrsfs, mathtools}
\usepackage{lmodern}
\usepackage{qsymbols}
\usepackage{latexsym}
\usepackage[noadjust]{cite}
\usepackage{pdfsync}
\usepackage{bm}
\usepackage{enumitem}

\newtheorem{thm}{Theorem}[section]
\newtheorem{lem}[thm]{Lemma}
\newtheorem{prop}[thm]{Proposition}

\theoremstyle{definition}
\newtheorem{defn}[thm]{Definition}

\numberwithin{equation}{section}
\usepackage[plainpages=false,pdfpagelabels,backref=page,citecolor=red]{hyperref}
\usepackage{xcolor}
\hypersetup{
colorlinks,
linkcolor={cyan!90!black},
citecolor={magenta},
urlcolor={green!40!black}
} %important to load after amsthm and before the rest
\newcommand{\R}{\mathbb{R}}
\newcommand{\IC}{\mathbb{C}}

\newcommand{\IZ}{\mathbb{Z}}

\newcommand{\ID}{\mathbb{D}}
\newcommand{\cH}{\mathcal{H}}
\newcommand{\cF}{\mathcal{F}} %Fourier transform

\renewcommand{\P}{\mathcal{P}}
\newcommand{\cE}{\mathcal{E}}
\newcommand{\loc}{\operatorname{loc}}
\renewcommand{\L}{\operatorname{L}} %Lebesgue spaces
\renewcommand{\H}{\operatorname{H}}
\newcommand{\Lloc}{\L_{\operatorname{loc}}} %local Lebesgue spaces
\newcommand{\C}{\operatorname{C}} %spaces of continuous functions
\renewcommand{\H}{\operatorname{H}} %Sobolev spaces, spaces of holomorphic functions
\newcommand{\E}{\mathrm{E}} %Energy space
\newcommand{\I}{\mathrm{I}}
\newcommand{\II}{\mathrm{II}}
\newcommand{\ree}{{\mathbb{R}^{n+1}}}
\newcommand{\gradx}{\nabla_x}
\renewcommand{\div}{\operatorname{div}}
\newcommand{\divx}{\div_x}
\newcommand{\dhalf}{D_t^{1/2}} %half time derivative
\newcommand{\HT}{H_t} %Hilbert transform
\newcommand{\B}{B}
\renewcommand{\i}{\mathrm{i}} %imaginary unit
\renewcommand{\d}{\, \mathrm{d}} %differential
\newcommand{\eps}{\epsilon} %epsilon
\renewcommand\Re{\operatorname{Re}}
\renewcommand\Im{\operatorname{Im}}
\newcommand{\cl}[1]{\overline{#1}} %closure
\DeclareMathOperator{\dom}{\mathsf{D}} %domain
\DeclareMathOperator\Max{\mathcal{M}} %HL Max Operator
\newcommand{\sgn}{\operatorname{sgn}}
\def\Xint#1{\mathchoice
{\XXint\displaystyle\textstyle{#1}}%
{\XXint\textstyle\scriptstyle{#1}}%
{\XXint\scriptstyle\scriptscriptstyle{#1}}%
{\XXint\scriptscriptstyle%
\scriptscriptstyle{#1}}%
\!\int}
\def\XXint#1#2#3{{\setbox0=\hbox{$#1{#2#3}{%
\int}$ }
\vcenter{\hbox{$#2#3$ }}\kern-.6\wd0}}
\def\barint{\,\Xint -} % \, corrects the \! used in the definition
\def\bariint{\barint_{} \kern-.4em \barint}
\def\bariiint{\bariint_{} \kern-.4em \barint}
\renewcommand{\iint}{\int_{}\kern-.34em \int} %\, minor space between the integrals
\renewcommand{\iiint}{\iint_{}\kern-.34em \int} %\, minor space between the integrals
\title[The Kato square root problem for weighted parabolic operators]{The Kato square root problem for\\ weighted parabolic operators}

\author{Alireza Ataei}
\email{alireza.ataei@math.uu.se}
\address{Department of Mathematics, Uppsala University, S-751 06 Uppsala,
Sweden}

\author{Moritz Egert}
\email{egert@mathematik.tu-darmstadt.de}
\address{TU Darmstadt, Fachbereich Mathematik, Schlossgartenstr.\ 7, 64289 Darmstadt, Germany}

\author{Kaj Nystr\"{o}m}
\email{kaj.nystrom@math.uu.se}
\address{Department of Mathematics, Uppsala University, S-751 06 Uppsala,
Sweden}

\thanks{}

\makeatletter
\@namedef{subjclassname@2020}{%
	\textup{2020} Mathematics Subject Classification}
\makeatother

 %Fractional derivatives and integrals
\date{\today}

\begin{document}
\begin{abstract}
We give a simplified and direct proof of the Kato square root estimate for parabolic operators with elliptic part in divergence form and coefficients possibly depending on space and time in a merely measurable way. The argument relies on the nowadays classical reduction to a quadratic estimate and a Carleson-type inequality. The precise organization of the estimates is different from earlier works. In particular, we succeed in separating space and time variables almost completely despite the non-autonomous character of the operator. Hence, we can allow for degenerate ellipticity dictated by a spatial $A_2$-weight, which has not been treated before in this context.
\end{abstract}

\maketitle
%\tableofcontents

%%%%%%%%%%%%%%%%%%%%%%%%%%%%%%%%%%%%%%%%%%%%%%%%%%%%%%%%%%%%%%%%%%%%%%%%%%%%
\section{Introduction and main result}
 In the variables $(x,t) \in \R^n \times \R \eqqcolon \ree$, we consider parabolic operators of the form
\begin{eqnarray}\label{eq1deg+}
\cH u \coloneqq \partial_t u  -w^{-1} \div_{x} (A \nabla_{x}u),
\end{eqnarray}
where the weight $w = w(x)$ is time-independent and belongs to the spatial Muckenhoupt class $A_2(\mathbb R^{n},\d x)$, and the coefficient matrix $A = A(x,t)$ is measurable with complex entries and possibly depends on all variables. Degeneracy is dictated by the same weight $w$ in the sense that $w^{-1}A$ satisfies the classical uniform ellipticity condition (Section~\ref{the coefficients}).

Weighted parabolic operators as in \eqref{eq1deg+} occur in many different contexts and applications, including the study of fractional powers of parabolic operators \cite{LN} and heat kernels of Schr{\"o}dinger equations with singular potential \cite{IKO}.  For contributions to
the study of local properties of solutions to  $\cH u = 0$ and Gaussian estimates, we refer to \cite{CS,CUR2}.

The purpose of this paper is to establish the Kato (square root) estimate for $\cH$, that is, to prove Theorem \ref{thm:Kato} stated below. We write $\L^2_\mu = \L^2(\ree, \d w \d t)$, $\d\mu=\d w \d t=w(x)\d x \d t$, for the natural weighted Lebesgue space associated with $\cH$, and $\E_\mu$ for the corresponding first-order parabolic Sobolev space of functions $u$ such that the spatial gradient $\gradx u$, as well as the half-order time derivative $\dhalf u$, are in $\L^2_\mu$. For the sake of the introduction, an intuitive interpretation of these objects suffices. We turn to rigorous definitions in Section~\ref{sec: function spaces} below.

\begin{thm}\label{thm:Kato} The operator $\cH$ can be defined as a maximal accretive operator in $\L_\mu^2$ associated with an accretive sesquilinear form with domain $\E_\mu$. The domain of its unique maximal accretive square root is the same as the form domain, that is $\dom(\sqrt{\cH}) = \E_\mu$, and
	\begin{align*}
		\|\sqrt {\cH}\, u\|_{\L^2_\mu} \sim \|\nabla_x u\|_{\L^2_\mu}+ \| \dhalf u\|_{\L^2_\mu}  \qquad (u \in \E_\mu)
	\end{align*}
holds with an implicit constant that depends on the dimension, the ellipticity parameters of $A$ and the $A_2$-constant for $w$.
\end{thm}

The time derivative $\partial_t$ is a skew-adjoint operator and hence there are no lower bounds for the formal pairing $\Re \langle \cH u, u\rangle$ that contains derivatives in $t$. However, when the time variable describes the full real line, parabolic operators admit some `hidden coercivity' that can be revealed through the Hilbert transform $\HT$ in the $t$-variable. Indeed, splitting $\partial_t = \dhalf \HT \dhalf$, the sesquilinear form associated with \eqref{eq1deg+} over $\L^2_\mu$ is
\begin{align}
	\label{hidden coercivity_intro}
	\B(u,v) \coloneqq \iint_{\ree} w^{-1} A \gradx u \cdot \cl{\gradx v} +  \HT \dhalf u \cdot \cl{\dhalf v} \, \d w \d t \qquad (u,v \in \E_\mu),
\end{align}
and lower bounds including both time and space derivatives become available when taking $v = (1+\delta \HT)u$ with $\delta>0$ small. This observation is originally due to Kaplan~\cite{Kaplan} and has been re-discovered several times ever since, see \cite{HLe,N1, DZ} for example. Maximal accretivity of $\cH$ essentially follows from this observation, but to the best of our knowledge  an explicit statement, in the unweighted case $w=1$, only appeared much later in \cite{AE}. For the reader's convenience, we reproduce the full argument in our weighted setting in Section~\ref{sec: parabolic operator}. Being maximal accretive, $\cH$ admits a sectorial functional calculus and in particular a (unique) maximal accretive square root $\sqrt{\cH}$, see \cite{Haase, Kato} for background. This is how our main result should be understood.

The pursuit for the solution of the Kato problem for unweighted elliptic operators (finally completed in \cite{AHLMcT}) introduced new techniques that proved themselves extremely viable for extensions and applications to other problems in harmonic analysis and partial differential equations~\cite{ Auscher-Axelsson, Auscher-Rosen, Auscher-Mourgoglou, MorrisWeightedElliptic, Li, Escauriaza, Auscher-Amenta, HKMP, 5Man, AEN1, CNS, N2}. For this reason, Kato-type estimates for different operators are desirable and the results of this paper most surely have important implications for, and applications to, boundary value problems for weighted second-order parabolic operators.

Let us mention that the case of $A_2$-weighted elliptic operators was settled in \cite{CUR}, see also \cite{CUR-Chema} for an extension, and re-discovered in the more general framework of first-order Dirac operators in \cite{ARR}. The third author was first to develop the underlying harmonic analysis in the unweighted parabolic setting in \cite{N1}, and in the same paper he proved the square function estimates that are essentially equivalent to Theorem~\ref{thm:Kato} when $w\equiv 1$ and when the coefficients $A$ are $t$-independent. Using a framework of parabolic Dirac operators, Auscher, together with the second and the third author, obtained the unweighted parabolic case when coefficients depend measurably on $x$ \emph{and} $t$~\cite{AEN}. Our Theorem~\ref{thm:Kato} completes this succession of results but there is more to it and that makes,  as we shall discuss next, the present paper interesting even in the unweighted case.

Under the assumption $A = A(x)$ in \cite{N1}, the operator $\cH$ is an autonomous parabolic operator and in retrospect, the main result in \cite{N1} could have been obtained by interpolation from maximal regularity of the Cauchy problem for \eqref{eq1deg+}, see \cite{Ouhabaz}. (In fact, this argument requires only smoothness of order ${1}/{2}$ for the coefficients in the $t$-variable.)  However, many of the techniques in \cite{N1} had already been strong enough for proving the parabolic Kato estimate in the presence of measurable $t$-dependence and our proof of Theorem~\ref{thm:Kato} shows exactly how, thereby making our result novel in at least two ways:

\begin{itemize}
	\item We generalize all previous findings in the parabolic setting by combining measurable dependence of the coefficients on all variables with $A_2$-weighted degeneracy in space.\\[-7pt]
	\item We avoid the Dirac operator framework in \cite{AEN}. The resulting `second-order' approach for parabolic operators with time-dependent measurable coefficients has not appeared in the literature before, and, when restricted to the unweighted case $w \equiv 1$, it provides a significant simplification of the proof of \cite[Thm.~2.6]{AEN}.
\end{itemize}
Our ambition is to present an almost self-contained argument using only a minimal number of tools. We do not attempt to generalize all further results in \cite{AEN} to the weighted setting, which should be done by developing a parabolic weighted Dirac operator framework.

As is customary in the field, the first reduction in the proof of Theorem~\ref{thm:Kato} is to use the bounded $\H^\infty$-calculus for maximal accretive operators and a duality argument in order to reduce the matter to the one-sided quadratic estimate
\begin{align}
\label{QE_intro}
	\int_0^\infty \|\lambda \cH (1+\lambda^2 \cH)^{-1} u\|_{\L^2_\mu}^2 \, \frac{\d \lambda}{\lambda} \lesssim  \|\nabla_x u\|_{\L^2_\mu}^2 + \| \dhalf u\|_{\L^2_\mu}^2  \qquad (u \in \E_\mu).
\end{align}
In contrast to the elliptic setting, this reduction does not follow immediately from classical results à la Kato and Lions, since the sesquilinear form $B$ in \eqref{hidden coercivity_intro} is \emph{not} closed due to the lack of lower bounds by half-order time derivatives. Some more care is needed but we settle the issue in Section~\ref{sec: reduction QE}. The quadratic estimate \eqref{QE_intro} is then achieved by slightly refining the techniques in \cite{N1} and the argument relies on (weighted) Littlewood--Paley theory in $\L^2$ (Section~\ref{sec: weighted LP}), which eventually reduces matters to a Carleson measure estimate that can be proved through a $Tb$-procedure (Section~\ref{sec2}).

It came as a surprise to us that even though coefficients may depend measurably on all variables, the proof of \eqref{QE_intro} can be arranged in a way that almost completely separates time and space variables. This observation incarnates in three different stages of the proof:
\begin{itemize}
	\item At the level of Littlewood--Paley theory, it suffices to use weighted elliptic theory in $x$ and classical Fourier analysis in $t$. \\[-7pt]
	\item At the level of off-diagonal bounds (averaged `kernel' bounds, see Section~\ref{subsec: OD}), we only need estimates for operators involving differentiation in space. These estimates can be deduced directly from the equation and come with parabolic scaling. The much more involved off-diagonal decay and Poincaré inequalities for non-local derivatives such as $\dhalf$, which were fundamental novelties in~ \cite{AEN}, can be avoided. \\[-7pt]
	\item At the level of the $Tb$-argument, the test functions  can be constructed based on a product structure, which makes the argument more straightforward compared to the system of functions used in \cite{AEN}.
\end{itemize}
These three observations have in common that we can re-group derivatives of resolvents of $\cH$ in such a way that fine harmonic analysis estimates need only to be applied to the spatial parts, whereas the $t$-derivatives appear in blocks that are amenable for simple resolvent estimates in $\L^2_\mu$-norm. We shall indicate the most striking examples of this principle along with the proof of the Carleson measure estimate in the final section.

The next section contains some preliminary notations and conventions. The rest of the paper follows the outline above.
%%%%%%%%%%%%%%%%%%%%%%%%%%%%%%%%%%%%%%%%%%%%%%%%%%%%%%%%%%%%%%%%%%%%%%%%%%%%
\section{Preliminaries and basic assumptions}
\label{sec: preliminaries}

Given $(x,t)\in\mathbb R^{n}\times\mathbb R$, we let $\|(x,t)\|:=\max \{|x|, |t|^{1/2}\}$. We call $\|(x,t)\|$
the parabolic norm of $(x,t)$. Given a half-open cube $Q = (x-r/2,x+r/2]^n \subset\mathbb R^n$ parallel to the coordinate axes with sidelength $r$ and an interval $I=(t-r^2/2,t+r^2/2]$, we call $\Delta:=Q\times I\subset\mathbb R^{n+1}$ a parabolic cube of size $r$. Occasionally, we write $\Delta_r(x,t) = Q_r(x) \times I_r(t)$ and $r=\ell(\Delta)$ to indicate the center and size directly. A dyadic parabolic cube of size $2^j$ is by definition centered in $(2^j \IZ)^n \times (4^j \IZ)$. For every $c>0$,  and given $\Delta$, we define $c \Delta$ as the parabolic cube with the same center as $\Delta$ and size $c \ell(\Delta)$.

\subsection{Assumptions and notations concerning the weight}
\label{the weight}
For general background and the results cited here, we refer to \cite[Ch.~V]{Stein}.  The weight $w=w(x)$ is a real-valued function belonging to the Muckenhoupt class $A_2(\mathbb R^{n},\d x)$, that is,
\begin{equation}
	\label{A2}
[w]_{A_2} \coloneqq \sup_Q \bigg(\barint_{Q} w\, \d x\bigg)\,
\bigg(\barint_{Q} w^{-1} \, \d x \bigg) < \infty,
\end{equation}
where the supremum is taken with respect to all cubes $Q\subset\mathbb R^n$. We introduce the measure $\d w(x):=w(x) \d x$ on $\mathbb R^{n}$ and we write $w(E) = \int_{E} \d w$ for all Lebesgue measurable sets $E \subset \R^n$. For averages, we use the notation
 $$(g)_{E,w} :=\barint_{E} g(x)\, \d w(x):= \frac{1}{w(E)}\int_{E} g(x)\, w(x) \d x$$
if $w(E)\in (0,\infty)$ and $g$ is locally integrable on $\R^n$ with respect to $\d w(x)$. It follows from \eqref{A2} that there are constants $\eta \in (0,1)$ and $\beta>0$, depending only on $n$ and $[w]_{A_2}$, such that
\begin{align}\label{ainfw}
	\beta^{-1} \biggl (\frac {|E|}{|Q|}\biggr )^{\frac{1}{2 \eta}} \leq \frac {w(E)}{w(Q)}\leq \beta \biggl (\frac {|E|}{|Q|}\biggr )^{2\eta},
\end{align}
whenever $E\subset Q$ is measurable and where $|\cdot|$ denotes Lebesgue measure in $\R^n$.  In particular, there exists a constant $D$ only depending on $[w]_{A_{2}}$ and $n$,  called doubling constant for $w$, such that
\begin{align}\label{doublingc}
	w(2Q)\leq Dw(Q) \mbox{ for all cubes } Q\subset\mathbb R^n.
\end{align}
The measures
\begin{equation}
\begin{aligned}
\label{measure}
\d\mu&=\d\mu(x,t):= w(x) \d x\d t,\\
\d\mu^{-1}&=\d\mu^{-1}(x,t):= w(x)^{-1} \d x\d t,
%:=\frac 1{w(x)} \d x\d t,
\end{aligned}
\end{equation}
are defined on $\ree$.  Naturally, $\mu, \mu^{-1}$ can be seen as measures on $\ree$ defined by $A_2(\mathbb R^{n+1},\d x\d t)$ weights and in the context of these measures we use similar notation as above. The doubling constant for $\mu$ with respect to parabolic scaling is $4D$.

\subsection{Maximal functions}
\label{maximal functions}
We introduce the maximal operators in the individual variables
\begin{align*}
	\mathcal{M}^{(1)}(g_1)(x) &:= \sup_{r>0 }{\barint_{Q_r(x)} |g_1|\, \d x},\\
	\mathcal{M}^{(2)}(g_2)(t) &:= \sup_{r>0 }{\barint_{I_r(t)} |g_2|\, \d t},
\end{align*}
for all locally integrable functions $g_1$ and $g_2$ on $\R^{n}$ and $\R$, respectively. The operator $\mathcal{M}^{(1)}$ is bounded on the weighted space $\L^2(\R^n, \d w)$ with a bound depending on $[w]_{A_2}$ and $n$ \cite[Thm.1, p.201]{Stein}. Both maximal operators can be naturally extended to $\L^2(\R^{n+1}, \d \mu)$ by keeping one of the variables fixed, and they are bounded in this setting.

\subsection{Assumptions on the coefficients}
\label{the coefficients}
The matrix-valued function  $$A=A(x,t)=\{A_{i,j}(x,t)\}_{i,j=1}^{n}$$ is assumed to have complex measurable entries $A_{i,j}$ that satisfy
\begin{equation}
\label{ellip}
 c_1|\xi|^2w(x)\leq \Re (A(x,t) \xi \cdot \cl{\xi}), \qquad
 |A(x,t)\xi\cdot\zeta|\leq c_2w(x)|\xi| |\zeta|
\end{equation}
for some $c_1, c_2\in (0,\infty)$ and for all $\xi,\zeta \in \IC^{n}$, $(x,t) \in \R^{n+1}$. Here, $u\cdot v=u_1 v_1 +...+u_{n} v_n$ and $\bar u$ denotes the complex conjugate of $u$ so that $u\cdot \cl{v}$ is the standard inner product on $\IC^{n}$. We refer to $c_1, c_2$ as the ellipticity constants of $A$. Assumption \eqref{ellip} is equivalent to saying that $w^{-1}A$ satisfies the classical uniform ellipticity condition.

\subsection{Floating constants} We refer to $n$ and the constants $[w]_{A_{2}}$, $c_1$, $c_2$,  appearing in \eqref{A2} and \eqref{ellip}, as structural constants. For $A,B\in (0,\infty)$, the notation $A\lesssim B$  means that $A \leq cB$ for some $c$ depending at most on the structural constants. The notations $A\gtrsim B$ and $A \sim B$ should be interpreted similarly.
%%%%%%%%%%%%%%%%%%%%%%%%%%%%%%%%%%%%%%%%%%%%%%%%%%%%%%%%%%%%%%%%%%%%%%%%%%%%
\section{Weighted function spaces}
\label{sec: function spaces}

  In this section we give a brief account of the relevant weighted function spaces. We let $\L^2_w = \L^2(\R^n, \d w)$ be the Hilbert spaces of square integrable functions with respect to $\d w$. Its norm is denoted by $\|\cdot\|_{2,w}$, its inner product by $\langle \cdot, \cdot \rangle_{2,w}$, and the operator norm of linear operators on that space by $\|\cdot\|_{2\to2, w}$. Thanks to the $A_2$-condition, we have
\begin{align}\label{embedd}
	\L^2_w \subset \Lloc^1(\R^n, \d x),
\end{align}
and the class $\C_0^\infty(\R^n)$ of smooth and compactly supported test functions is dense in $\L^2_w$ via the usual truncation and convolution procedure~\cite[Sec.~1]{Kilp}. The same notations and properties apply to $\L^2_\mu$ in $\ree$.

\begin{defn}[Elliptic weighted Sobolev space]
\label{def: elliptic Sobolev space}
We write $\H_w^{1} := \H_w^{1}(\R^n)$ for the space of all $f \in \L_w^2$ for which the distributional gradient $\gradx f$ is (componentwise) in $\L_w^2$, and we equip the space with the norm $\|\cdot\|_{\H_w^1} \coloneqq (\|\cdot\|_{2,w}^2 + (\|\gradx \cdot\|_{2,w}^2)^{1/2}$.
\end{defn}

By construction $\H^1_w$ is a Hilbert space and standard truncation and convolution techniques yield that $\C_0^\infty(\R^n)$ is dense in $\H_w^{1}$, see~\cite[Thm.~2.5]{Kilp}.

In order to define parabolic function spaces, we use the Fourier transform $\cF$ in the time-variable, keeping in mind that if $f \in \L^2(\ree, \d \mu)$, then $f(x,\cdot) \in \L^2(\R, \d t)$ for a.e.\ $x \in \R^n$. The corresponding Fourier variable will be denoted by $\tau$. Then,
\begin{align*}
	\HT f \coloneqq \cF^{-1}(\i \sgn(\tau) \cF f)
\end{align*}
is our Hilbert transform. If $|\tau|^{1/2} \cF f \in \L^2_\mu$, then we define the half-order time derivative
\begin{align*}
	\dhalf f \coloneqq \cF^{-1}(|\tau|^{1/2} \cF f),
\end{align*}
and this is what we mean when we write $\dhalf f \in \L_\mu^2$. Using a classical formula for fractional Laplacians, see \cite{Hitchhiker} for example, we obtain for a.e.\ fixed $x \in \R^n$,
\begin{align}
\label{difference quotient formula}
\|\dhalf f\|_{2,\mu}^2 = \frac{2}{\pi} \int_{\R^n} \int_\R \int_\R \frac{|f(x,t) - f(x,s)|^2}{|t-s|^{2}} \, \d s \d t \d w(x),
\end{align}
with the right-hand side being finite precisely when $\dhalf f \in \L^2_\mu$.

\begin{defn}[Parabolic energy space]
\label{def: energy space}
We write $\E_\mu \coloneqq \E_\mu(\ree)$ for the space of all $f \in \L^2_\mu$ for which $\gradx f , \dhalf f \in \L^2_\mu$, and we equip the space with the norm
\begin{align*}
	\|\cdot\|_{\E_\mu} \coloneqq (\|\cdot\|_{2,\mu}^2 + \|\gradx \cdot\|_{2,\mu}^2 + \|\dhalf \cdot\|_{2,\mu}^2)^{\frac{1}{2}}.
\end{align*}
For $f \in \E_\mu$, we will refer to the vector $\ID f \coloneqq (\gradx f, \dhalf f)$ as the parabolic gradient of $f$.
\end{defn}

Again, $\E_\mu$ is a Hilbert space. Note that in the unweighted setting of \cite{N1} the notation $\ID$ has a slightly different meaning.

\begin{lem}
\label{lem: energy approximation}
The following statements are true:
	\begin{enumerate}
		\item The space $\C_0^\infty(\ree)$ is dense in $\E_\mu(\ree)$.
		\item Multiplication by $\C^1(\R^{n+1})$-functions is bounded on $\E_\mu(\ree)$.
	\end{enumerate}
\end{lem}

\begin{proof}
We begin with (i). If $f \in \E_\mu$, then convolutions with smooth mollifiers, separately in $x$ and $t$, provide smooth approximations in $\E_\mu$. For the convolution in space, this argument uses the $A_2$-condition on $w$ as mentioned above. Hence, it suffices to approximate $f$ by compactly supported functions in $\E_\mu$. To this end, we can follow the standard pattern of smooth truncation: We pick a sequence $(\eta_j)_j \subset \C_0^\infty(\ree)$ such that $\eta_j \to 1$ pointwise a.e. as $j \to \infty$, $\|\eta_j\|_\infty + j \|\gradx \eta_j\|_\infty + j \|\partial_t \eta_j\|_\infty \leq c$ uniformly in $j$, and then we set $f_j \coloneqq \eta_j f$. By dominated convergence, we obtain $f_j \to f$ and $\gradx f_j \to \gradx f$ in $\L_\mu^2$ as $j \to \infty$. For the half-order derivative, we use \eqref{difference quotient formula} with $f_j-f$ in place of $f$. We first bound the integrand in \eqref{difference quotient formula} by
\begin{align}
\label{eq1: energy approximation}
\begin{split}
\frac{|(f_j-f)(x,t) - (f_j-f)(x,s)|^2}{|t-s|^{2}}
&\leq 2 \frac{|(\eta_j-1)(x,t) - (\eta_j-1)(x,s)|^2}{|t-s|^{2}} |f(x,t)|^2 \\
&\quad +  2 \frac{|f(x,t) - f(x,s)|^2}{|t-s|^{2}} |(\eta_j-1)(x,s)|^2 \\
&\leq 2 \min \bigg\{ c^2, \frac{4 (c+1)^2}{|t-s|^2} \bigg\} |f(x,t)|^2 \\
&\quad +  2 (c+1)^2 \frac{|f(x,t) - f(x,s)|^2}{|t-s|^{2}}.
\end{split}
\end{align}
The right-hand side is independent of $j$ and integrable with respect to $\!\d s \d t \d w(x)$. Since the middle term tends to $0$ a.e.\ as $j \to \infty$, we conclude $\|\dhalf (f_j-f)\|_{2,\mu} \to 0$ by dominated convergence. This completes the proof of (i).

As for (ii), we note that if $\eta \in \C^1(\ree)$ and $f \in \E_\mu$, then
\begin{align*}
	\|\eta f\|_{2,\mu} & \leq \|\eta\|_\infty \|f\|_{2,\mu}, \\
	\|\gradx(\eta f)\|_{2,\mu} & \leq \|\eta\|_\infty \|\gradx f\|_{2,\mu} + \|\gradx \eta\|_\infty \|f\|_{2,\mu}, \\
	\|\dhalf(\eta f)\|_{2,\mu} & \leq \sqrt{8} \|\eta\|_\infty^{\frac{1}{2}} \|\partial_t \eta\|_\infty^{\frac{1}{2}} \|f\|_{2,\mu} + \|\eta\|_\infty \|\dhalf f\|_{2,\mu},
\end{align*}
where the third line follows by the same splitting as in \eqref{eq1: energy approximation}, but with $\eta$ in place of $1-\eta_j$.
\end{proof}

Lemma~\ref{lem: energy approximation} (i) implies the chain of continuous and dense embeddings
\begin{align}
\label{Gelfand}
	\E_\mu \subset \L^2_\mu \simeq (\L^2_\mu)^* \subset (\E_\mu)^*,
\end{align}
where we use the upper star to denote (anti)-dual spaces. We have bounded operators
\begin{align}
\label{E differentials}
\begin{split}
	\dhalf&: \E_\mu \to \L^2_\mu, \\
	\nabla_x&: \E_\mu \to (\L^2_\mu)^n, \\
\end{split}
\end{align}
and we denote their adjoints with respect to \eqref{Gelfand} by
\begin{align}
\label{Estar differentials}
\begin{split}
	\dhalf&: \L^2_\mu \to (\E_\mu)^*,\\
	w^{-1} \div_x w &: (\L^2_\mu)^n \to (\E_\mu)^*.\\
\end{split}
\end{align}
Note carefully that $w^{-1} \div_x w$ is only a suggestive notation reflecting the formal action of this operator. In general, there is no guarantee that this operator splits into a composition of its three building blocks.
%However, in the space $w^{-1} C^{\infty}_0(\R^{n+1}):=\{w^{-1} f:\ f\in C^{\infty}_0(\R^{n+1})\}$, the operator acts as $(w^{-1} \div_x w) (w^{-1} f) = w^{-1} \div_x f \in \L^2_{\mu}$ for all $f \in C^{\infty}_0(\ree).$
%%%%%%%%%%%%%%%%%%%%%%%%%%%%%%%%%%%%%%%%%%%%%%%%%%%%%%%%%%%%%%%%%%%%%%%%%%%%
\section{The parabolic operator}
\label{sec: parabolic operator}

 We continue by introducing the formal parabolic operator in \eqref{eq1deg+} rigorously as an unbounded operator in the Hilbert space $\L^2_\mu$ associated with a sesquilinear form.

Denoting by $\HT$ the Hilbert transform in the $t$-variable and by $\dhalf$ the half-order time derivative as defined in Section~\ref{sec: function spaces}, we can factorize
\begin{align*}
	\partial_t = \dhalf \HT \dhalf.
\end{align*}
By \eqref{E differentials} and \eqref{Estar differentials}  we have $\partial_t: \E_\mu \to (\E_\mu)^*$. We define $\cH$ as a bounded operator $\E_\mu \to (\E_\mu)^*$ via
\begin{align}
\label{hidden coercivity}
(\cH u)(v) \coloneqq \B(u,v) \coloneqq \iint_{\ree} w^{-1} A \gradx u \cdot \cl{\gradx v} +\HT \dhalf u \cdot \cl{\dhalf v} \, \d \mu \qquad (u,v \in \E_\mu).
\end{align}
In view of \eqref{Gelfand} it makes sense to consider the maximal restriction of $\cH$ to an operator in $\L^2_\mu$, called the part of $\cH$ in $\L^2_\mu$, with domain
\begin{align}\label{domain}
	\dom(\cH) := \{u \in \E_\mu(\ree) : \cH u \in \L_\mu^2(\ree) \}.
\end{align}
If $u \in \dom(\cH)$, we have for all $v \in \E_\mu$ that
\begin{align*}
	(\cH u)(v) = \iint_\ree \cH u \cdot \cl{v} \, \d \mu
\end{align*}
and a formal integration by parts in \eqref{hidden coercivity} reveals that it is indeed justified to say that the part of $\cH$ in $\L^2_\mu$ gives a meaning to the formal expression in \eqref{eq1deg+}. More precisely, in terms of \eqref{E differentials} and \eqref{Estar differentials}, we have that
\begin{align}
\label{correct H factorization}
	\cH = \dhalf \HT \dhalf - (w^{-1} \div_xw)(w^{-1}A \nabla_x).
\end{align}

\subsection{Hidden coercivity}
The following lemma relies on Kaplan's~\cite{Kaplan} hidden coercivity of the parabolic sesquilinear form $\B$ in \eqref{hidden coercivity} that can be revealed through the Hilbert transform.

\begin{lem}
\label{lem: hidden coercivity} Let $\sigma \in \IC$ with $\Re \sigma > 0$. For each $f \in (\E_\mu)^*$ there exists a unique $u \in \E_\mu $ such that $(\sigma + \cH) u = f$. Moreover,
\begin{align*}
	\|u\|_{\E_\mu} \leq \sqrt{2} \max \Big \{\frac{c_2 +1}{c_1}, \frac{|\Im \sigma| + 1}{\Re \sigma} \Big\} \|f\|_{(\E_\mu)^*}.
\end{align*}
\end{lem}

\begin{proof}
By Plancherel's theorem, the Hilbert transform $\HT$ is isometric on $\E_\mu$. Hence, we can define a `twisted' sesquilinear form $\B_{\delta,\sigma}: \E_\mu \times \E_\mu \to \IC$ via
\begin{align}
\label{eq: twisted sesquilinear form}
\begin{split}
 \B_{\delta,\sigma}(u,v)
	&\coloneqq \iint_{\ree} \biggl (\sigma u \cdot \cl{(1+\delta \HT)v} + w^{-1} A \gradx u \cdot \cl{\gradx (1+\delta \HT)v} \\
	&\qquad  \quad +\HT \dhalf u \cdot \cl{\dhalf (1+\delta \HT)v}\biggr) \, \d \mu,
\end{split}
\end{align}
where $\delta \in (0,1)$ is to be chosen.  Clearly $\B_{\delta,\sigma}$ is bounded. Since $\HT$ is skew-adjoint, we have
\begin{align}
\label{eq: skew adjoint}
\Re \iint_\ree \HT v  \cdot \cl{v} \d \mu = 0 \qquad(v \in \L^2_\mu).
\end{align}
Expanding $\B_{\delta,\sigma}(u,u)$ and using the above along with the weighted ellipticity of the coefficients $A$, we find
\begin{align}
\label{eq0: hidden coercivity}
	\Re  \B_{\delta,\sigma}(u,u)
	\geq \delta \|\dhalf u\|_{2,\mu}^2 + (c_1 - c_2\delta)\|\nabla_x u\|_{2,\mu}^2 + (\Re \sigma - \delta |\Im \sigma|) \|u\|_{2,\mu}^2.
\end{align}
Choosing $\delta = \min \bigl(\frac{c_1}{c_2+1}, \frac{\Re \sigma}{ |\Im \sigma|+1}\bigr)$, the factors in front of the second and third term in the last display become no less than $\delta$. Hence, we obtain the coercivity estimate
\begin{align}
\label{eq1: hidden coercivity}
	\Re \B_{\delta,\sigma} (u,u) \geq \min \Big\{\frac{c_1}{c_2 + 1}, \frac{\Re \sigma}{|\Im \sigma| + 1}\Big \} \|u\|_{\E_\mu}^2 \qquad (v \in \E_\mu).
\end{align}
The Lax-Milgram lemma yields, for each $f \in (\E_\mu)^*$,  a unique $u \in {\E_\mu}$ satisfying the estimate claimed in the lemma such that
\begin{align*}
\B_{\delta,\sigma}(u,v) = f((1+\delta \HT)v) \qquad (v \in {\E_\mu}).
\end{align*}
(Note that the additional factor $\sqrt{2}$ is an upper bound for the norm of $1+ \delta \HT$ on $\E_\mu$.)
Plancherel's theorem yields that $1+\delta \HT$ is an isomorphism on ${\E_\mu}$ for all $\delta\in\mathbb R$. Thus,
\begin{align*}
	\iint_{\ree} \sigma u \cdot \cl{v} + w^{-1} A \gradx u \cdot \cl{\gradx v} +\HT \dhalf u \cdot \cl{\dhalf v} \, \d \mu = f(v) \qquad  (v \in {\E_\mu}),
\end{align*}
that is, $(\sigma + \cH) u  = f$ as required.
\end{proof}

The proof above fails for $\delta =0$ since $\Re \B(\cdot, \cdot)$ does not control $\|\dhalf \cdot\|_{2,\mu}$ from above. As a consequence, $\B$ itself is \emph{not} a closed sesquilinear form in the sense of Kato~\cite{Kato}, or equivalently, $(\|\cdot\|_{2,\mu}^2 + \Re B(\cdot,\cdot))^{1/2}$ does \emph{not} define an equivalent norm on $\E_\mu$. In \cite[Lem.~4]{AE}, it has been (essentially) shown that a parabolic analog of Kato's first representation theorem holds nonetheless. For convenience, we include the short argument with some minor improvements in the next section.

\subsection{Maximal accretivity}
Recall that an operator in a Hilbert space such as $\L^2_\mu$ is called \emph{maximal accretive} if it is closed and densely defined, with resolvent estimates
\begin{align*}
	\|(\sigma + \cH)^{-1}\|_{2 \to 2,\mu} \leq (\Re \sigma)^{-1} \qquad (\sigma \in \IC, \Re \sigma > 0).
\end{align*}

\begin{prop}
\label{prop: max accretive}
The part of $\cH$ in $\L^2_\mu$ is maximal accretive and $\dom(\cH)$ is dense in $\E_\mu$.
\end{prop}

\begin{proof}
Fix $\sigma \in \IC$ with $\Re \sigma > 0$. Lemma~\ref{lem: hidden coercivity} yields that $\sigma + \cH: \dom(\cH) \to \L^2_\mu$ is bijective. Given $f \in \L^2_\mu$, we set $u \coloneqq (\sigma + \cH)^{-1} f$ and use ellipticity of the coefficients $A$ and \eqref{eq: skew adjoint} to deduce
\begin{align*}
	\Re \sigma \|u\|_{2,\mu}^2
	&\leq \Re \iint_{\R^{n+1}} \sigma u \cdot \cl{u} + w^{-1} A\nabla_x u \cdot \cl{\nabla_x u} + \HT \dhalf u \cdot \cl{\dhalf  u}\, \d \mu \\
	&= \Re \iint_{\R^{n+1}} f \cdot\cl{u}\,\d \mu \leq \|f\|_{2,\mu} \|u\|_{2,\mu}.
\end{align*}
This gives the required resolvent bound $\|u\|_{2,\mu} \leq (\Re \sigma)^{-1} \|f\|_{2,\mu}$. Moreover, the part of $\cH$ in $\L^2_\mu$ is closed since it has non-empty resolvent set, and the resolvent estimates for $\sigma > 0$ imply dense domain~\cite[Prop.~2.1.1]{Haase}. This proves maximal accretivity.

In order to prove that $\dom(\cH)$ is dense in $\E_{\mu}$, we use the sesquilinear form $B_{\delta,1}$ in \eqref{eq: twisted sesquilinear form} with  $\delta > 0$ chosen as in the proof of that lemma. Suppose $v \in \E_\mu$ is orthogonal to $\dom(\cH)$ in $\E_\mu$. By the Lax-Milgram lemma there is $w \in \E_\mu$ such that $\langle u,v \rangle_{\E_\mu} = B_{\delta,1}(u,w)$ for all $u \in \E_\mu$. For $u \in \dom(H)$ this identity becomes $0 = \langle (1+ \cH)u, (1+ \delta \HT)w \rangle_{2,\mu}$ and since $1+ \cH: \dom(\cH) \to \L_\mu^2$ is bijective, we conclude that $(1+ \delta \HT)w = 0$. Thus, we have $w=0$ and therefore also $v=0$.
\end{proof}

The adjoint $\cH^*$ of $\cH$ (seen as either a bounded operator $\E_\mu \to (\E_\mu)^*$ or an unbounded operator in $\L^2_\mu$) has the same properties as $\cH$. Indeed, it can be checked by the very definition that it is associated with the sesquilinear form
\begin{align*}
	B^*(u,v) = \cl{B(v,u)}
\end{align*}
and that it formally corresponds to the backward-in-time operator
\begin{align*}
	-\partial_t -w^{-1}(x)\div_{x} (A^*(x, t)\nabla_{x}).
\end{align*}
Here $A^*$ is the conjugate transpose of $A$.

\subsection{Resolvent estimates} Using Proposition \ref{prop: max accretive}, we see that for $\lambda > 0$ the resolvent operators
\begin{equation}
\begin{aligned}
\label{resolvents}
	\mathcal{E}_\lambda&:=(I+\lambda^2\cH)^{-1}, \\
	\mathcal{E}_\lambda^\ast&:=(I+\lambda^2\cH^\ast)^{-1}
 \end{aligned}
\end{equation}
are well-defined as bounded operators $\L^2_\mu \to \L^2_\mu$ and $(\E_\mu)^* \to \E_\mu$. Moreover, they are adjoints of each other.

\begin{lem}
\label{le8-}
The following resolvent estimates hold uniformly for all $\lambda>0$, all $f\in \L^2_\mu$ and all $\mathbf{f}\in (\L^2_\mu)^n$:
\begin{align*}
		\mathrm{(i)}&\ \|\mathcal{E}_\lambda f\|_{2,\mu}+\|\lambda {\mathbb D}\mathcal{E}_\lambda f\|_{2,\mu}\lesssim \| f\|_{2,\mu},\notag\\
		\mathrm{(ii)}&\  \|\lambda \mathcal{E}_\lambda \dhalf f \|_{2,\mu} + \|\lambda^2 {\mathbb D}\mathcal{E}_\lambda \dhalf f \|_{2,\mu}\lesssim \|{f}\|_{2,\mu},\notag\\
		\mathrm{(iii)}&\  \|\lambda \mathcal{E}_\lambda w^{-1}\div_x(w\mathbf{f})\|_{2,\mu} + \|\lambda^2{\mathbb D} \mathcal{E}_\lambda w^{-1}\div_x(w\mathbf{f}) \|_{2,\mu}\lesssim \|\mathbf{f}\|_{2,\mu}.
\end{align*}
The same estimates hold with $\mathcal{E}_\lambda$ replaced by $\mathcal{E}_\lambda^\ast$.
\end{lem}
\begin{proof}
We first prove (i). Setting $u \coloneqq (\lambda^{-2} + \cH)^{-1}f$, we have $\cE_\lambda f = \lambda^{-2} u$ and by maximal accretivity we obtain
\begin{align*}
	\|\cE_\lambda f\|_{2,\mu} \leq \|f\|_{2,\mu}.
\end{align*}
Next, we use the twisted sesquilinear form $B_{\delta,\sigma}$ as in \eqref{eq: twisted sesquilinear form} with parameter $\sigma = \lambda^{-2}$, so that by construction
\begin{align}
\label{eq1: le8-}
	B_{\delta,\sigma}(u, u) = \langle f, (1+\delta \HT)u \rangle_{2,\mu}.
\end{align}
With this choice for $\sigma$, we pick $\delta = c_1/(2c_2)$, use \eqref{eq0: hidden coercivity} on the left, and Cauchy--Schwarz on the right, in order to obtain
\begin{align*}
	\|\ID u\|_{2,\mu}^2 \lesssim \|f\|_{2,\mu} \|u\|_{2,\mu} \leq \lambda^2 \|f\|_{2,\mu}^2.
\end{align*}
This is the required uniform bound for $\lambda {\mathbb D}\mathcal{E}_\lambda f$. Since $\cH$ is of the same type as $\cH^*$ from the point of view of sesquilinear forms, the same estimates also hold for $\cE_\lambda^*$ in place of $\cE_\lambda$.

Next, we note that the estimates for the leftmost terms in (ii) and (iii) follow by duality from (i) applied to $\cE_\lambda^*$.

In order to estimate the second term on the left in (ii), we set $u \coloneqq (\lambda^{-2} + \cH)^{-1} \dhalf f$. Since $\dhalf f$ is now regarded as an element in $(\E_\mu)^*$, we get $\langle f, \dhalf  (1+\delta \HT)u \rangle_{2,\mu}$ on the right-hand side in \eqref{eq1: le8-}, and from this we conclude
\begin{align*}
 \|\ID u\|_{2,\mu}^2 \lesssim \|f\|_{2,\mu} \|\ID u\|_{2,\mu}
\end{align*}
as required. The remaining term in (iii) is estimated in the same way upon replacing $\dhalf f$ by $w^{-1} \divx(w \mathbf{f})$.
\end{proof}

\subsection{Off-diagonal estimates}
\label{subsec: OD} Given measurable subsets $E,F$ of $\ree$, we let $$ d{(E,F)}:= \inf\{\|(x-y,t-s)\|: (x,t) \in E, (y,s) \in F\}$$ denote their parabolic distance. Lemma~\ref{le8-+} below is an improvement of the uniform bounds in Lemma~\ref{le8-}. We only state and prove Lemma \ref{le8-+} for the families of operators we need it for later on. However, let us stress that such estimates are not to be expected in the presence of the non-local operator $\dhalf$ and one of the insights in \cite{AEN} was that in this case a non-local version of off-diagonal bounds should be used.

\begin{lem}\label{le8-+}
Assume that $E,F$ are measurable subsets of $\mathbb{R}^{n+1}$ and let $d \coloneqq d(E,F).$ Then there exists a constant $c \in (0,\infty)$, depending only on the structural constants, such that
\begin{align*}
	\mathrm{(i)}&\ \iint_{F} |\mathcal{E}_\lambda f|^2+|\lambda \nabla_x \mathcal{E}_\lambda f|^2\, \d\mu \lesssim  e^{-\frac{d}{c\lambda}} \iint_{E} |f|^2\,\d\mu,\notag\\
	\mathrm{(ii)}&\  \iint_{F}|\lambda\mathcal{E}_\lambda w^{-1}\div_x(w\mathbf{f})|^2 \, \d\mu \lesssim e^{-\frac{d}{c\lambda}} \iint_{E} |\mathbf{f}|^2\, \d \mu,
\end{align*}
for all $\lambda > 0$ and all $f \in \L^2_{\mu}$, $\mathbf{f} \in (\L^2_{\mu})^n$ with support in $E$. The same statements are true when $\mathcal{E}_\lambda$ is replaced by $\mathcal{E}_\lambda^\ast$.
\end{lem}
\begin{proof}
As in the  proof of Lemma \ref{le8-}, it suffices to treat $\cE_\lambda$. Based on Lemma $\ref{le8-}$, we see that it suffices to obtain the exponential estimate for $\lambda \leq \alpha d$, where for now $\alpha \in (0,1)$ is a degree of freedom that will be determined later and which will only depend on the structural constants.

Let $u \coloneqq \mathcal{E}_{\lambda} f$ and recall that
\begin{align}
	\label{eq1: le8-+}
	\iint_{\mathbb{R}^{n+1}}  u \, \cl{v} + \lambda^2 w^{-1}A \nabla_x u \cdot \cl{\nabla_x v} + \lambda^2 \HT\dhalf u \cdot \cl{\dhalf v}\, \d \mu = \iint_{\mathbb{R}^{n+1}} f \cdot \cl{v}\,   \d \mu
\end{align}
for all $v \in \E_\mu$. We can pick a real-valued $\tilde{\eta}\in \C^{\infty}(\mathbb{R}^{n+1})$ such that $\tilde{\eta} =1$ on $F$, $\tilde{\eta} = 0$ on $E$, and such that
$$d|\gradx \tilde{\eta}| + d^2|\partial_t \tilde{\eta}|  \leq c$$
for some constant $c$ only depending on $n$. The different scaling in the two terms is due to the definition of the parabolic distance. Next, we let
\begin{align}
	\label{eq2: le8-+}
	v \coloneqq u \eta^2 \quad \text{with} \quad \eta \coloneqq e^{\frac{\alpha d}{\lambda} \tilde{\eta}}-1.
\end{align}

For this choice of $v$, we rewrite the real part in \eqref{eq1: le8-+} of the pairing containing half-order derivatives as follows. According to Lemma~\ref{lem: energy approximation}, there exists a sequence $\{u_i\} \subset \C^{\infty}_0(\ree)$ such that $u_i \to u$ in $\E_\mu$ as $i \to \infty$. By the same lemma, $\eta^2 u_i \to \eta^2u$ in $\E_\mu$ and therefore
\begin{align*}
	\Re \iint_{\ree}  \HT\dhalf u \cdot \cl{\dhalf v}\, \d \mu
	& = \Re \lim_{i \to \infty} \iint_{\ree} \HT\dhalf u_i \cdot \cl{\dhalf (u_i \eta^2) }\, \d t \d w\\
	& =\lim_{i \to \infty} \Re \iint_{\ree} \partial_t u_i \cdot \cl{ u_i \eta^2 }\, \d t \d w\\
	& = \frac{1}{2} \lim_{i \to \infty} \iint_{\ree}  \partial_t |u_i|^2 \cdot \eta^2\, \d t \d w \\
	& = \frac{1}{2}\lim_{i \to \infty}-  \iint_{\ree}  |u_i|^2 \cdot \partial_t (\eta^2)\, \d t \d w\\
	&=- \frac{1}{2} \iint_{\ree} |u|^2 \cdot \partial_t (\eta^2)\, \d \mu.
\end{align*}
Going back to \eqref{eq1: le8-+} and using that $\eta = 0$ on $E$, we conclude that
\begin{align*}
	\Re \iint_{\mathbb{R}^{n+1}} |u|^2 \, \eta^2\, \d \mu + \lambda^2 w^{-1}A \nabla_x u \cdot \cl{\nabla_x (u \eta^2)} - \frac{\lambda^2}{2} |u|^2 \partial_t (\eta^2)\, \d \mu = 0.
\end{align*}
Using this identity and ellipticity, we deduce
\begin{align*}
	&\iint_{\ree}  |u|^2 \, \eta^2\, \d \mu + c_1 \lambda^2 \iint_{\ree} |\nabla_x u|^2 \, \eta^2\, \d \mu \\
	& \leq  \lambda^2 \iint_{\ree} |u|^2 \,|\eta| \,|\partial_t \eta| \, \d \mu + 2c_2 \lambda^2 \iint_{\ree}  |u| \, |\nabla_x u|\, |\eta|\, |\nabla_x \eta|\, \d \mu \\
	& \leq \frac{1}{2} \iint_{\ree} |u|^2 \, \eta^2 \, \d \mu + \frac{1}{2}\lambda^4 \iint_{\ree} |u|^2\, |\partial_t \eta|^2\, \d \mu\\
	&\quad+\frac{1}{2} c_1 \lambda^2  \iint_{\ree} |\nabla_x u|^2 \, \eta^2\, \d \mu+ 2 \frac{c_{2}^2}{c_1} \lambda^2 \iint_{\ree} |u|^2 \, |\nabla_x \eta|^2\, \d \mu.
\end{align*}
In conclusion,
\begin{align*}
	\iint_{\ree}  |u|^2 \, \eta^2\, \d \mu + c_1 \lambda^2 \iint_{\ree} |\nabla_x u|^2 \, \eta^2\, \d \mu
	\leq  \iint_{\ree}  |u|^2 \biggl(\lambda^4 |\partial_t \eta|^2 + 4\frac{c_2^2}{c_1} \lambda^2 |\nabla_x \eta|^2\biggr)\, \d \mu.
\end{align*}
By the definition of $\eta$ in \eqref{eq2: le8-+}, and since $\lambda \leq  \alpha d \leq d$, we see that
\begin{align*}
	|\partial_t \eta|^2 \leq  \frac{\alpha^2 d^2}{\lambda^2} |\eta+1|^2 \frac{c^2}{d^4} \leq  c^2 \alpha^2 \lambda^{-4} |\eta+1|^2
\end{align*}
and
\begin{align*}
	|\gradx \eta|^2 \leq  \frac{\alpha^2 d^2}{\lambda^2} |\eta+1|^2 \frac{c^2}{d^2} =   c^2 \alpha^2 \lambda^{-2} |\eta+1|^2.
\end{align*}
Thus, we get
\begin{align*}
	\iint_{\ree} |u|^2 \, \eta^2\, \d \mu + c_1 \lambda^2 \iint_{\ree} |\nabla_x u|^2 \, \eta^2\, \d \mu \lesssim  \alpha^2 \iint_{\ree} |u|^2 |\eta+1|^2\, \d \mu.
\end{align*}
At this point, we make our choice of $\alpha$. Indeed, using the bound $|\eta +1|^2 \leq 2(\eta^2 + 1)$, we choose $\alpha$ small enough to be able to absorb the part coming from $\eta$ into the left-hand side. The conclusion is that
\begin{align*}
	\iint_{\ree} |u|^2 \, \eta^2\, \d \mu + \lambda^2 \iint_{\ree} |\nabla_x u|^2 \, \eta^2\, \d \mu \lesssim \iint_{\ree} |u|^2\, \d \mu.
\end{align*}
On the right-hand side, we can use Lemma \ref{le8-} (i), and, on the left-hand side, we exploit that on $F$ we have $$\eta =e^{\frac{\alpha  d}{\lambda}}-1 \geq \frac{1}{2} e^{\frac{\alpha  d}{\lambda}},$$ since we are assuming $\lambda \leq \alpha d$. Consequently,
\begin{align*}
	& e^{\frac{2 \alpha  d}{\lambda}} \iint_{F} |u|^2\, \d \mu  +  e^{\frac{2 \alpha  d}{\lambda}}\iint_{F} |\lambda \nabla_x u|^2 \, \d \mu \lesssim \iint_{E}  |f|^2\, \d \mu,
\end{align*}
which proves (i).

The inequality in (ii) follows by a duality argument, using (i) for $\mathcal{E}_{\lambda}^*$ and  interchanging the roles of $E$ and $F$. In fact,
\begin{align*}
	\iint_{F}|\lambda\mathcal{E}_\lambda w^{-1}\div_x(w\mathbf{f})|^2 \, \d\mu
	&= \sup_{g} \biggl(\iint_{\ree} \lambda\mathcal{E}_\lambda w^{-1}\div_x(w{\bf{f}}) \cdot \cl{g}\, \d \mu\biggr)^2 \\
	&= \sup_{g}\biggl( \iint_{E} - {\bf{f}}\cdot \cl{\lambda  \nabla_x \mathcal{E}_\lambda^*{g}}\, \d \mu \biggr)^2,
\end{align*}
where the supremum is taken with respect to all $g \in \L^2_\mu$,  with support in $F$, such that $\|g\|_{2,\mu}=1$. We can now complete the proof by applying the Cauchy--Schwarz inequality and (i) of the lemma but for $\mathcal{E}_\lambda^*$.
\end{proof}
%%%%%%%%%%%%%%%%%%%%%%%%%%%%%%%%%%%%%%%%%%%%%%%%%%%%%%%%%%%%%%%%%%%%%%%%%%%%
\section{Weighted Littlewood--Paley theory in the parabolic setting}
\label{sec: weighted LP}

 We could develop a weighted parabolic Littlewood--Paley theory following the approach for singular integrals on spaces of homogeneous type~\cite{DJS}. However, since  our weight $w$ is time independent, we decided to present a down-to-earth approach by combining weighted elliptic theory known in the field~ \cite{CUR1, CUR2, Garcia-Cuerva} with Fourier analysis on the real line. Most of our estimates here are formulated using the square function norm
\begin{align}
\label{SF norm}
 |||\cdot|||_{2,\mu}:=\biggl (\int_0^\infty\iint_{\mathbb R^{n+1}}|\cdot|^2\, \frac{\d \mu\d\lambda}\lambda\biggr )^{\frac{1}{2}}.
\end{align}

For the rest of the paper, $\P\in \C_0^\infty(\ree)$ is a fixed real-valued function  in product form
$$\P(x,t)=\P^{(1)}(x)\P^{(2)}(t),$$
where $\P^{(1)}$ and $\P^{(2)}$ are both radial, non-negative, and have integral $1$. For all $x \in \R^n, t \in \R$, we set
\begin{align*}
    &\P_\lambda^{(1)}(x) := \lambda^{-n} \P^{(1)}(x/\lambda), \\
    &\P_\lambda^{(2)}(t) := \lambda^{-2} \P^{(2)}(t/\lambda^2), \\
    & \P_\lambda(x,t):=\P_\lambda^{(1)}(x)\P_\lambda^{(2)}(t)=\lambda^{-n-2}\P^{(1)}(x/\lambda) \P^{(2)}(t/\lambda^2),
\end{align*}
whenever $\lambda>0$.
% We will need the additional cancellation property
%\begin{align} \label{momentum}
%   \int_{\R^n}  x_{i} \, \P^{(1)}(x) \, \d x =0,
%\end{align}
%for all integers $1 \leq i \leq n$.
With a slight abuse of notation, we let $\P_\lambda$ also denote the associated convolution operator
$$\P_\lambda f(x,t)=\P_\lambda\ast f(x,t)=\iint_{\mathbb R^{n+1}}\P_\lambda(x-y,t-s)f(y,s)\, \d y\d s,$$
and likewise for $\P_\lambda^{(1)}$ and $\P_\lambda^{(2)}$. We note that
\begin{equation}
\begin{aligned}\label{dda}
|\P^{(1)}_\lambda f(x,t)| & \leq  \mathcal{M}^{(1)}(f(\cdot,t))(x),\\ |\P^{(2)}_\lambda f(x,t)| & \leq  \mathcal{M}^{(2)}(f(x,\cdot))(t), \\ |\P_\lambda f(x,t)| & \leq  \mathcal{M}^{(1)}(\mathcal{M}^{(2)}f(\cdot,t))(x),
\end{aligned}
\end{equation}
almost everywhere, for every $f \in \L^1_{\loc}(\ree)$, see \cite[Sec.~II.2.1]{Stein}. In particular, these pointwise bounds hold for $f \in \L^2_\mu$.
The boundedness of the maximal operators in $\L^2_\mu$ implies
\begin{align*}
	\sup_{\lambda > 0}  \|\P_{\lambda}\|_{2 \to 2,\mu} \lesssim 1,
\end{align*}
see Section~\ref{maximal functions}.
%\begin{lem}
%     \label{/}
%     Let $f \in \L^2_{\mu}(\ree)$. Then,
%     \begin{align*}
%        \mathrm{(i)}& \quad \|\P_{\lambda}\|_{2 \to 2,\mu} \lesssim 1, \\
%        \mathrm{(ii)}& \quad  \lim_{\lambda \to 0} \|\P_{\lambda} f - f\|_{2,\mu} = 0.
%     \end{align*}
%The same properties hold for $P_{\lambda}^{(1)}$ and $P_{\lambda}^{(2)}$ in place of $P_{\lambda}$.
%\end{lem}
%\begin{proof}
%Part (i) follows from \eqref{dda} and the boundedness of the maximal operators, see subsection \ref{maximal functions}. Using that $f \in \Lloc^1(\R^{n+1})$ we get that all three families ($P_{\lambda}^{(1)}$, $P_{\lambda}^{(2)}$ and  $P_{\lambda}$) are pointwise almost everywhere approximations of the identity, and hence (ii) follows also from \eqref{dda} by dominated convergence.
%\end{proof}

\begin{lem}\label{little1} For all $f\in \L^2_\mu(\mathbb R^{n+1})$,
\begin{align*}
	|||\lambda\nabla_x \P_\lambda f|||_{2,\mu}+|||\lambda^2\partial_t\P_\lambda f|||_{2,\mu}+|||\lambda \dhalf \P_\lambda f|||_{2,\mu}\lesssim \|f\|_{2,\mu}.
\end{align*}

\end{lem}
\begin{proof}
Here, we write out in detail how the product structure of $\P_\lambda$ can be used to prove parabolic estimates in $\R^{n+1}$ through weighted elliptic theory and classical Fourier analysis. This motif will appear in all proofs of this section. Let $\hat{g}$ denote the Fourier transform  in time of a function $g$ on $\R^{n+1}$.

By uniform boundedness of $\P^{(1)}_\lambda$ in $\L^2_\mu$ and Plancherel's theorem, we have
\begin{align*}
    |||\lambda \dhalf \P_\lambda f|||^2_{2,\mu}   &= \int_0^{\infty} \iint_{\ree} | \P^{(1)}_\lambda \lambda \dhalf \P^{(2)}_\lambda f|^2 \, \frac{\d \mu \d \lambda}{\lambda} \\ & \lesssim \int_{\R^n} \int_0^{\infty} \int_{\R} | \lambda \dhalf \P^{(2)}_\lambda f|^2 \, \frac{\d t \d \lambda}{\lambda}  \d w \\ & = \int_{\R^n} \int_0^{\infty} \int_{\R} | \lambda |\tau|^{\frac{1}{2}} \widehat{\P^{(2)}}(\lambda^2 \tau) \hat{f}(x,\tau)|^2 \, \frac{\d \tau \d \lambda}{\lambda}  \d w(x).
\end{align*}
The integral in $\lambda$ is finite and independent of $\tau$ since $\widehat{\P^{(2)}}$ is a radial Schwartz function. Applying Plancherel's theorem backwards, we get the desired bound by $\|f\|_{2,\mu}^2$. The same argument yields the bound for $|||\lambda^2 \partial_t \P_\lambda f|||_{2,\mu}$.

Finally, in order to bound $\lambda \nabla_x \P_\lambda f$, we use uniform boundedness of $\P^{(2)}_\lambda$ to get
\begin{align*}
	|||\lambda\nabla_x \P_\lambda f|||^2_{2,\mu}
	& = \int_0^{\infty} \iint_{\ree} | \P^{(2)}_\lambda \lambda \nabla_x \P^{(1)}_\lambda f|^2 \, \frac{\d \mu \d \lambda}{\lambda} \\
	& \lesssim \int_\R \int_0^{\infty} \int_{\R^n} | \lambda \nabla_x \P^{(1)}_\lambda f|^2 \, \frac{\d w \d \lambda}{\lambda} \d t.
\end{align*}
For fixed $t$, we now need weighted elliptic Littlewood-Paley theory. The operator $\lambda \nabla_x \P^{(1)}_\lambda$ acts by convolution with $\Psi_\lambda$, where $\Psi = \nabla_x \P^{(1)}$ has integral $0$. Thus, we can use e.g.\ \cite[Lem.~4.6]{CUR1} to control the integral in $\d w \d \lambda$ by $\|f(\cdot,t)\|_{2,w}^2$ and the proof is complete.
\end{proof}

\begin{lem}
\label{little2}
For all $f \in \E_{\mu}$,
\begin{eqnarray*}
|||\lambda^{-1}(I-\P_\lambda) f|||_{2,\mu}\lesssim \|{\mathbb D} f\|_{2,\mu}.
\end{eqnarray*}
\end{lem}

\begin{proof}
We first claim
\begin{align}\label{ML}
|||\lambda^{-1}(I-\P_\lambda^{(1)}) f|||_{2,\mu}+|||\lambda^{-1}(I-\P_\lambda^{(2)}) f|||_{2,\mu}\lesssim \|{\mathbb D} f\|_{2,\mu}.
\end{align}
As in the proof of Lemma \ref{little1}, this can be proved using Plancherel's theorem in $t$ for the second term, and weighted Littlewood-Paley theory with $t$ fixed for the first term. The required weighted result is \cite[Prop.~2.3]{CUR} (originally \cite[Prop.~4.7]{CUR1}) and the application to the concrete operator considered here is detailed in the lines following equation (4.3) in the same paper.

In order to complete the proof of the lemma, we simply write
\begin{align*}
(I-\P_\lambda)=\P_\lambda^{(2)}(1-\P_\lambda^{(1)})+(1-\P_\lambda^{(2)}).
\end{align*}
The result follows from \eqref{ML} and the uniform boundedness of $\P^{(2)}_\lambda$ in $\L^2_\mu$.
\end{proof}

In the following we write $\Delta = Q \times I$ for parabolic cubes in $\R^{n+1} = \R^n \times \R$.
\begin{defn}
We define $\mathcal{A}_\lambda^{(1)}, \mathcal{A}_\lambda^{(2)}, \mathcal{A}_\lambda$ to be the dyadic averaging operators in $x$, $t$ and $(x,t)$ with respect to parabolic scaling, that is, if $\Delta = Q \times I$ is the dyadic parabolic cube with $\ell(\Delta)/2 < \lambda \leq \ell(\Delta)$ containing $(x,t)$, then
\begin{align*}
\mathcal{A}_\lambda^{(1)} f(x,t) &:= \barint_{Q}f\, \d y, \\
\mathcal{A}_\lambda^{(2)} f(x,t) &:= \barint_{I}f\, \d s,  \\
      \mathcal{A}_\lambda f(x,t) &:=\bariint_{\Delta}f\, \d y\d s = \mathcal{A}_\lambda^{(1)} \mathcal{A}_\lambda^{(2)}  f(x,t).
\end{align*}
\end{defn}

It follows from the bounds for the maximal operators in Section~\ref{maximal functions} and doubling that the dyadic averaging operators are bounded on $\L^2_\mu$, uniformly in $\lambda$.

\begin{lem} \label{little3} Let  $\P_\lambda$ and $\mathcal{A}_\lambda$ be as above. Then, for all $f\in \L^2_\mu(\mathbb R^{n+1})$,
\begin{eqnarray*}
 |||(\mathcal{A}_\lambda-\P_\lambda)f|||_{2,\mu}\lesssim \|f\|_{2,\mu}.
\end{eqnarray*}
\end{lem}

\begin{proof}
We follow our (general) strategy and write
\begin{align*}
	\mathcal{A}_\lambda - \P_\lambda = 	\mathcal{A}_\lambda^{(2)}(\mathcal{A}_\lambda^{(1)} - 	\P_\lambda^{(1)}) + \P_\lambda^{(1)}(\mathcal{A}_\lambda^{(2)} - 	\P_\lambda^{(2)}),
\end{align*}
where we have also used that $	\mathcal{A}_\lambda^{(2)}$ and $\P_\lambda^{(1)}$ commute since they act in different variables. Since these operators are uniformly bounded on $\L^2_\mu$ with respect to $\lambda$, we get
\begin{align*}
	 |||(\mathcal{A}_\lambda-\P_\lambda)f|||_{2,\mu}
	 &\lesssim \int_\R \int_0^\infty \|(\mathcal{A}_\lambda^{(1)} - 	\P_\lambda^{(1)})f(\cdot,t)\|_{2,w}^2 \, \frac{\d \lambda}{\lambda} \d t \\
	 &\quad + \int_{\R^n} \int_0^\infty \|(\mathcal{A}_\lambda^{(2)} - 	\P_\lambda^{(2)})f(x,\cdot)\|_{2,\d t}^2 \, \frac{\d \lambda}{\lambda} \d w (x).
\end{align*}
For the first term on the right, we can rely on the weighted elliptic version of the lemma~\cite[Lem.~5.2]{CUR1}. For the second term on the right, we can make a change of variables $\lambda' = \lambda^2$ and use the unweighted one-dimensional version of the lemma, which  of course follows from the same reference or the classical proof in \cite[App.~C, Eq.~(4)]{Asterisque}.
\end{proof}
%%%%%%%%%%%%%%%%%%%%%%%%%%%%%%%%%%%%%%%%%%%%%%%%%%%%%%%%%%%%%%%%%%%%%%%%%%%%
\section{Reduction to a quadratic estimate}
\label{sec: reduction QE}

 The purpose of this short section is to reduce our main result, Theorem~\ref{thm:Kato}, to the quadratic estimate
\begin{eqnarray}
\label{ea1intro}
      |||\lambda \cH \cE_\lambda f|||_{2,\mu}\lesssim \|\nabla_x f\|_{2,\mu} +\|\HT \dhalf f \|_{2,\mu} \qquad (f \in \E_\mu).
\end{eqnarray}
Recall that $\cE_\lambda = (1+ \lambda^2 \cH)^{-1}$. Since the sesquilinear form associated with $\cH$ is not closed, see Section~\ref{sec: parabolic operator}, classical results \`a la Lions as in the elliptic case do not apply and here we give full details of this reduction.

At this point, we require some essentials from functional calculus. We give references along the way and we refer the reader to \cite{Mc, Haase} for further background. Since $\cH$ is maximal accretive (Proposition~\ref{prop: max accretive}), it has a unique maximal accretive square root $\sqrt{\cH}$ defined by the functional calculus for sectorial operators and the same is true for the adjoint $\cH^*$ with $\sqrt{\cH^*} = (\sqrt{\cH})^*$.

In order to see the reduction alluded to above, we start out with the Calderón reproducing formula in \cite[Thm.~5.2.6]{Haase} and we write
\begin{eqnarray}
\label{est1+kaintro}
\sqrt{\cH}f= \frac{16}{\pi} \int_0^\infty \lambda^3\cH^2 (1+\lambda^2\cH)^{-3} f\, \frac {\d\lambda}\lambda,
\end{eqnarray}
where $f \in \dom(\sqrt{\cH})$ and the integral is understood as an improper Riemann integral in $\L_\mu^2$. Testing this identity against $g \in \L^2_\mu$ and applying Cauchy--Schwarz, we obtain
\begin{eqnarray*}
%\label{est1+intro}
|\langle\sqrt{\cH}f,g\rangle_{2,\mu}|
\leq \frac{16}{\pi}|||\lambda \cH(1+\lambda^2\cH)^{-1} f|||_{2,\mu}\times|||\lambda^2\cH^\ast
       (1+\lambda^2\cH^\ast)^{-2}g|||_{2,\mu}.
\end{eqnarray*}
The second term is controlled by a structural constant times $\|g\|_{2,\mu}$ since $\cH^*$ is maximal accretive in $\L^2_\mu$ --- more precisely, this follows from von Neumann's inequality~\cite[Thm.~7.1.7]{Haase} and the characterization of the emerging functional calculus due to McIntosh~\cite[Thm.~7.3.1]{Haase}. Taking the supremum over all $g$ yields
\begin{align*}
    \|\sqrt{\cH}f\|_{2,\mu} \lesssim |||\lambda \cH(I+\lambda^2\cH)^{-1} f|||_{2,\mu}.
\end{align*}
Let us now suppose that \eqref{ea1intro} holds. Then, we obtain
\begin{align}
\label{ga1}
 \| \sqrt{\cH}f \|_{2,\mu} \lesssim \|\nabla_x f\|_{2,\mu} +\|\HT \dhalf f \|_{2,\mu},
\end{align}
when $f$ is in  $\E_\mu \cap \dom(\sqrt{\cH}) \supset \dom(\cH)$. However, since this space is dense in $\E_\mu$ by Proposition~\ref{prop: max accretive}, and as $\sqrt{\cH}$ is closed, the estimate extends to all $f \in \E_\mu$. Next, we note that $\cH^*$ is similar to an operator in the same class as $\cH$ under conjugation with the `time reversal' $f(t,x) \mapsto f(-t,x)$ and conjugation of $A$. Hence, we also have
\begin{align}
\label{ga2}
 \| \sqrt{\cH^*}g \|_{2,\mu} \lesssim \|\nabla_x g\|_{2,\mu} +\|\HT \dhalf g \|_{2,\mu},
\end{align}
whenever $g \in \E_\mu$. Using \eqref{eq0: hidden coercivity} with $\sigma = 0$ and $\delta$ small enough depending on the structural constants, we obtain for all $f \in \dom(\cH)$ that
\begin{align*}
\delta \|\nabla_x f \|^2_{2,\mu} + \delta \|\dhalf f\|^2_{2,\mu}
&\leq |\langle \cH f,(1+\delta \HT) f\rangle_{2,\mu}| \\
%&= |\langle \sqrt{\cH} f,\sqrt{\cH^\ast} (1+\delta \HT) g\rangle| \\
&\leq \|\sqrt{\cH}f \|_{2,\mu} \|\sqrt{\cH^*} (1+\delta \HT) f \|.
\end{align*}
Now, \eqref{ga2} with $g := (1+\delta \HT)f \in \E_\mu$ implies
\begin{align}
\label{ga3}
\|\nabla_x f \|_{2,\mu} + \|\dhalf f \|_{2,\mu} \lesssim  \|\sqrt{\cH} f \|_{2,\mu}.
\end{align}
Since $\dom(\cH)$ is dense in $\dom(\sqrt{\cH})$ for the graph norm \cite[Prop.~3.1.1(h)]{Haase}, the estimate extends to all $f \in \dom(\sqrt{\cH})$.

In conclusion, we have seen that \eqref{ea1intro} implies the statement of Theorem~\ref{thm:Kato} through the estimates \eqref{ga1} and \eqref{ga3}. Therefore, the rest of the paper is devoted to the task of proving \eqref{ea1intro}.
%%%%%%%%%%%%%%%%%%%%%%%%%%%%%%%%%%%%%%%%%%%%%%%%%%%%%%%%%%%%%%%%%%%%%%%%%%%%
\section{Principal part approximation}
\label{sec: principal part approximation}

 In order to prove the square function estimate \eqref{ea1intro}, we will eventually split $\cH$ into its elliptic and parabolic parts and perform the `hard' analysis only on the elliptic part. This will lead us to the operators
\begin{equation}\label{stand-rpaa}
 \mathcal{U}_\lambda :=\lambda\mathcal{E}_\lambda w^{-1}\div_xw \quad (\lambda > 0).
\end{equation}
These operators appeared in Lemma~\ref{le8-+} on off-diagonal estimates and in particular they are uniformly bounded on $(\L^2_\mu)^n$. Here we continue their analysis.

Given a cube $Q=Q_r(x)\subset\R^n$ and an interval $I=I_r(t)$, we let $\Delta:=Q\times I$ and set
 \begin{align*}
C_{k}(\Delta)=C_{k}(Q \times I)&:= 2^{k+1} \Delta \setminus 2^{k} \Delta \quad (k=1,2,\ldots),\\
C_0(\Delta) &:= 2 \Delta.
\end{align*}
In the following, we denote the characteristic function of a set $E$ by $1_E$. We use off-diagonal estimates to define $\mathcal{U}_\lambda$ on $(\L^\infty)^n$.

\begin{defn}
\label{welldef}
For $\mathbf{b} \in (\L^\infty)^n$ we define
\begin{equation}
\label{deffa}
\mathcal{U}_\lambda \mathbf{b} =: \lim_{k \to \infty} \mathcal{U}_\lambda (\mathbf{b} 1_{2^k \Delta}),
\end{equation}
with convergence locally in $({\L}^2_{\mu})^n$, where on the right $\Delta$ is any parabolic cube.
\end{defn}

Definition \ref{welldef} is meaningful and independent of the choice of $\Delta$ as we shall see next. To start, if  $\Delta'$ is any parabolic cube, then for $m > n$ large enough to guarantee that $\Delta' \subset 2^{n-1} \Delta$, Lemma~\ref{le8-+} yields
\begin{align*}
   \| \mathcal{U}_\lambda (\mathbf{b} 1_{2^m\Delta  \setminus 2^n \Delta}) \|_{\L^2_{\mu}(\Delta')}
   &\leq \sum_{j=n}^{m-1} \| \mathcal{U}_\lambda (\mathbf{b} 1_{C_{j}(\Delta) }) \|_{\L^2_{\mu}(\Delta')} \\
   & \lesssim \mu(\Delta)^{\frac{1}{2}}  \|\mathbf{b}\|_{\infty} \sum_{j=n}^{m} e^{-\frac{ \ell(\Delta) 2^{ j-1}}{c \lambda}} (4D)^{j+1}.
\end{align*}
Recall that $4D$ is the doubling constant for $\mu$, see \eqref{doublingc}. The right-hand side converges to $0$ as $m,n \to \infty$. In conclusion, $\{\mathcal{U}_\lambda (\mathbf{b} 1_{2^n\Delta})\}$ is a Cauchy sequence locally in $(\L^2_{\mu})^n$. By the same argument, Definition \ref{welldef}  is independent of of the particular choice $\Delta$. Taking $\Delta' = \Delta$ and $n=1$, we get
\begin{equation}
\begin{aligned}
    \label{ineqwell}
    \| \mathcal{U}_\lambda \mathbf{b} \|_{\L^2_{\mu}(\Delta)} & \leq  \| \mathcal{U}_\lambda (\mathbf{b} 1_{2\Delta}) \|_{\L^2_{\mu}(\Delta)} +  \| \mathcal{U}_\lambda (\mathbf{b} 1_{\ree \setminus 2 \Delta}) \|_{\L^2_{\mu}(\Delta)} \\ & \lesssim \mu(\Delta)^{\frac{1}{2}} \|\mathbf{b}\|_{\infty} \bigg( 1+ \sum_{j=1}^{\infty}  e^{-\frac{\ell(\Delta) 2^{ j-1}}{c \lambda}}(4D)^{j+1}\bigg).
\end{aligned}
\end{equation}

\begin{lem}\label{cor5.6} Let $\mathbf{b} \in (\L^\infty)^n$ and ${f}\in \L^2_\mu$. Then,
$$
\|(\mathcal{U}_\lambda \mathbf{b}) \mathcal{A}_{\lambda} {f}\|_{2,\mu} \lesssim \|\mathbf{b}\|_\infty \| {f}\|_{2,\mu}.
$$
\end{lem}

\begin{proof}
If $\Delta \subset \mathbb{R}^{n+1}$ is a parabolic cube such that $\ell(\Delta)/2 < \lambda \leq \ell(\Delta)$, then by \eqref{ineqwell} we have
\begin{align*}
\iint_{\Delta}|\mathcal{U}_\lambda \mathbf{b}|^2\, \d\mu\lesssim \mu(\Delta)\|\mathbf{b}\|^2_\infty.
\end{align*}
Since $\mathcal{A}_{\lambda}f$ is constant on each such $\Delta$, we obtain
\begin{align*}
 \iint_{\Delta} |(\mathcal{U}_\lambda \mathbf{b}) \mathcal{A}_{\lambda} {f}|^2 \, \d \mu
 \leq \iint_{\Delta} |\mathcal{U}_\lambda \mathbf{b}|^2\, \d\mu \cdot \bariint_{\Delta} |\mathcal{A}_{\lambda} {f}|^2\, \d\mu
 \lesssim   \|\mathbf{b}\|^2_\infty \iint_{\Delta} |\mathcal{A}_{\lambda} {f}|^2\, \d\mu.
\end{align*}
The claim follows by summing in $\Delta$ and using that $\mathcal{A}_\lambda$ is uniformly bounded on $\L^2_\mu$ with respect to $\lambda$, see Section~\ref{sec: weighted LP}.
\end{proof}

Writing $A=(A_1,\ldots, A_n)$ with $A_j$ the $j$-th column of $A$, we can use Definition~\ref{welldef} to define the action of $\mathcal{U}_\lambda$ on the bounded matrix-valued function $w^{-1}A$ by
\begin{align*}
	(\mathcal{U}_\lambda w^{-1}A) := \mathcal{U}_\lambda (w^{-1} A) := \Big(\mathcal{U}_{\lambda} (w^{-1} A_1),\cdot \cdot \cdot,\mathcal{U}_{\lambda}(w^{-1} A_n)\Big).
\end{align*}
We will approximate $\mathcal{U}_\lambda w^{-1} A$ by operators that act via multiplication on the maximal dyadic cubes of size at most $\lambda$. To be precise, we will consider
 \begin{align}\label{opa1}
\mathcal{R}_\lambda f := \mathcal{U}_\lambda (w^{-1}A f)   -(\mathcal{U}_\lambda w^{-1}A) \mathcal{A}_{\lambda}f.
 \end{align}
This is nowadays called the `principal part approximation'. Using Lemmas \ref{le8-} and \ref{cor5.6}, we see that the $\mathcal{R}_\lambda$ are uniformly bounded on $\L^2_\mu$ for $\lambda > 0$. Moreover, we prove the following bound.

\begin{prop}
\label{le10}
Let $f\in \L^2_\mu \cap \C^{\infty}$. Then,
\begin{align*}
	\|\mathcal{R}_\lambda f\|_{2,\mu}\lesssim \|\lambda \nabla_x f\|_{2,\mu}+\|\lambda^2 \partial_t f\|_{2,\mu}.
\end{align*}
\end{prop}

For the proof, we need the following weighted Poincar{\'e}-type inequality. In the following we abbreviate $(f)_\Delta = (f)_{\Delta, \d x \d t}$.

\begin{lem}
\label{lem: Poincare}
Let $f \in \C^\infty$. Then, for all parabolic cubes $\Delta$ and all non-negative integers $k$,
\begin{align*}
   \iint_{C_k(\Delta)} |(f- (f)_{\Delta})|^2 \, \d \mu
    \leq c (k+1) \iint_{2^{k+1}\Delta} 2^{2k} \ell(\Delta)^2 |\nabla_x f|^2 +2^{4k} \ell(\Delta)^4 |\partial_tf|^2 \, \d \mu,
\end{align*}
where $c$ depends only on $n$ and $[w]_{A_2}$.
\end{lem}

\begin{proof}
Let $\Delta = Q \times I$. We set $g:= (f)_{Q, \d x}$, which is a function of $t$, and we split
\begin{align*}
    f - (f)_{\Delta} = (f - (f)_{Q, \d x}) + (g - (g)_{I, \d t}).
\end{align*}
To the first term we can apply the weighted Poincar{\'e} inequality in the $x$-variable from (the proof of) \cite[Thm.~15.26]{HKM} and to the second term the standard Poincar{\'e} inequality in the $t$-variable. The result is
\begin{align*}
	\biggl(\iint_{\Delta} |(f- (f)_{\Delta}) |^2 \, \d \mu\biggr)^{\frac{1}{2}}
	\leq c \biggl(\iint_{\Delta} \ell(\Delta)^2 |\nabla_x f|^2 + \ell(\Delta)^4 |\partial_tf|^2 \, \d \mu\biggr)^{\frac{1}{2}}.
\end{align*}
Note that in \cite{HKM} balls are used instead of cubes, but doubling allows us to switch between one and the other. For the general result it suffices to write
\begin{align*}
    f - (f)_{\Delta} = \big(f - (f)_{2^{k+1}\Delta}\big)  + \big((f)_{2^{k+1}\Delta} - (f)_{2^{k}\Delta}\big) + \ldots + \big((f)_{2\Delta} - (f)_{\Delta}\big),
\end{align*}
and to use the estimate above on the cubes $2^{k+1}\Delta$ and $2^{k+1} \Delta, \ldots, 2 \Delta$.
%From the first part with $2^{k+1}\Delta$ in place of $\Delta$ we get
%\begin{align*}
%    \big\|\big(f-(f)_{2^{k+1}\Delta,\mu}\big)1_{C_k(\Delta)}\big\|_{2,\mu}
%    \lesssim 2^k \ell(\Delta) \|\nabla_x  f\|_{2,\mu} + 2^{2k} \ell(\Delta)^2 \| \partial_t  f\|_{2,\mu}.
%\end{align*}
%Similarly,
%\begin{align*}
%\big\|\big((f)_{2^{i+1}\Delta,\mu} - (f)_{2^{i}\Delta,\mu})1_{C_k(\Delta)}\big\|_{2,\mu}
%&= \mu(C_k(\Delta))^{\frac{1}{2}} \big|(f)_{2^{i+1}\Delta,\mu} - (f)_{2^{i}\Delta,\mu}\big| \\
%&\leq \mu(2^{k+1} \Delta)^{\frac{1}{2}} \bariint_{2^i \Delta} |f- (f)_{2^{i+1} \Delta, \mu}| \, \d \mu \\
%&\leq ((4D)^{k-i+1})^{\frac{1}{2}} \bigg(\iint_{2^{i+1} \Delta} |f- (f)_{2^{i+1} \Delta, \mu}|^2 \, \d \mu\bigg)^{\frac{1}{2}} \\
%&\lesssim ((4D)^{k-i+1})^{\frac{1}{2}} \bigg(2^i \ell(\Delta) \|\nabla_x  f\|_{2,\mu} + 2^{2i} \ell(\Delta)^2 \| \partial_t  f\|_{2,\mu} \bigg),
%\end{align*}
%where we have used doubling and H\"older's inequality in the second step. Summing up the previous estimates yields the claim.
\end{proof}

\begin{proof}[Proof of Proposition~\ref{le10}]
We note that if $(x,t) \in \R^{n+1}$ and $\lambda > 0$, then
\begin{align*}
	\mathcal{R}_\lambda f (x,t) = \mathcal{U}_\lambda  \Bigl(w^{-1} A(f- (f)_{\Delta})\Bigr) (x,t),
\end{align*}
where $\Delta$ is the unique dyadic parabolic cube with $\ell(\Delta)/2 < \lambda \leq \ell(\Delta)$ that contains $(x,t)$. Thus,
\begin{align*}
	\|\mathcal{R}_\lambda f\|_{2,\mu}^2
&= \sum_{\Delta} \iint_{\Delta} \Bigl| \mathcal{U}_\lambda  \Bigl(w^{-1} A(f- (f)_{\Delta})\Bigr) \Bigr|^2 \, \d \mu \\
&\leq \sum_{\Delta} \biggl(\sum_{k=0}^\infty \biggl( \iint_{\Delta} \Bigl| \mathcal{U}_\lambda \Bigl( w^{-1} A \cdot 1_{C_k(\Delta)}(f- (f)_{\Delta})\Bigr) \Bigr|^2 \, \d \mu \biggr)^{\frac{1}{2}}\biggr)^{2},
\end{align*}
and therefore
\begin{align*}
	\|\mathcal{R}_\lambda f\|_{2,\mu}^2
&\lesssim \sum_{\Delta} \biggl(\sum_{k=0}^\infty e^{-\frac{2^k}{c}} \biggl( \iint_{C_k(\Delta)} \Bigl|(f- (f)_{\Delta}) \Bigr|^2 \, \d \mu \biggr)^{\frac{1}{2}}\biggr)^{2} \\
&\lesssim \sum_{\Delta} \sum_{k=0}^\infty e^{-\frac{2^k}{c}}  \iint_{C_k(\Delta)} \Bigl|(f- (f)_{\Delta}) \Bigr|^2 \, \d \mu  \\
&\lesssim \sum_{\Delta} \sum_{k=0}^\infty e^{-\frac{2^k}{c}} (k+1) 2^{4k} \iint_{2^{k+1}\Delta} \lambda^2 |\nabla_x f|^2 + \lambda^4 |\partial_tf|^2 \, \d \mu \\
&\leq  \biggl(\sum_{k=0}^\infty e^{-\frac{2^k}{c}} (k+1) 2^{4k} 2^{(k+1)(n+2)}\biggr) \iint_{\R^{n+1}} \lambda^2 |\nabla_x f|^2 + \lambda^4 |\partial_tf|^2 \, \d \mu ,
\end{align*}
where we used, in succession, the off-diagonal estimates, Cauchy--Schwarz inequality, Lemma~\ref{lem: Poincare}, and the fact that each point in $\R^{n+1}$ is contained in exactly $2^{(k+1)(n+2)}$ of the cubes $2^{k+1} \Delta$. The sum in $k$ is still finite and the proof is complete.
%Let $\Delta$ be a parabolic cube satisfying $\lambda \leq \ell(\Delta)$.
%Then,
% \begin{align*}
%     \|\mathcal{R}_{\lambda} ( f)\|_{2,\mu}
%     & \lesssim \sum_{k=0}^{\infty}  \bigl|\bigl|\mathcal{R}_{\lambda} \bigl(( \bigl(f-(f)_{\Delta,\mu} 1_{C_k(\Delta)}\bigr) \bigr|\bigr|_{2,\mu} \\
%     & \lesssim \sum_{k=0}^{\infty} e^{-\frac{2^k \ell(\Delta)}{\lambda}} \bigl|\bigl|  \bigl(f- (f)_{\Delta, \mu} \bigr) 1_{C_k(\Delta)} \bigr|\bigr|_{2,\mu} \\
%     & \lesssim \|\lambda \nabla_x  f\|_{2,\mu} + \|\lambda^2 \partial_t  f\|_{2,\mu},
% \end{align*} where we used $\mathcal{R}_{\lambda} 1 =0$ and triangle inequality in the first inequality, \eqref{stand} in the second inequality, and Lemma~\ref{lem: Poincare} in the last inequality.
\end{proof}
%%%%%%%%%%%%%%%%%%%%%%%%%%%%%%%%%%%%%%%%%%%%%%%%%%%%%%%%%%%%%%%%%%%%%%%%%%%%
\section{Proof of Theorem \ref{thm:Kato}}
\label{sec2}

 After the reduction in Section~\ref{sec: reduction QE} it remains to prove the quadratic estimate \eqref{ea1intro} that we now write in the form
\begin{eqnarray}\label{kee}
      |||\lambda\mathcal{E}_\lambda \cH f|||_{2,\mu}\lesssim \|{\mathbb D}f\|_{2,\mu} \qquad (f \in \E_\mu).
 \end{eqnarray}
In the following we will use the operators $\P_\lambda$, $\mathcal{A}_\lambda$, $\mathcal{U}_\lambda$, $\mathcal{R}_\lambda$ that have been introduced in Section \ref{sec: parabolic operator}, ~\ref{sec: weighted LP} and \ref{sec: principal part approximation}. Collecting the estimates from these sections, we can at this stage prove the following.

\begin{prop}
\label{prop: upshot 5 and 7}
Let $f \in \E_\mu$. Then,
\begin{align*}
 |||(\lambda\mathcal{E}_\lambda \cH + (\mathcal{U}_\lambda w^{-1} A) \mathcal{A}_\lambda \nabla_x)f|||_{2,\mu}\lesssim \|{\mathbb D}f\|_{2,\mu}.
\end{align*}
\end{prop}

\begin{proof}
We begin by writing
\begin{align}
\label{eq1: essence}
\lambda\mathcal{E}_\lambda \cH f=\lambda\mathcal{E}_\lambda \cH\P_\lambda f+\lambda \cH \mathcal{E}_\lambda (I-\P_\lambda) f.
\end{align}
Using the identity
$$\lambda \cH \mathcal{E}_\lambda=\lambda^{-1}(I-\mathcal{E}_\lambda),$$
the uniform $\L^2_\mu$-boundedness of $\mathcal{E}_\lambda$, and Lemma \ref{little2}, we see that
\begin{eqnarray*}
|||\lambda\mathcal{E}_\lambda \cH(I-\P_\lambda) f|||_{2,\mu}\lesssim |||\lambda^{-1}(I-\P_\lambda)f|||_{2,\mu}\lesssim\|{\mathbb D} f\|_{2,\mu}.
\end{eqnarray*}
Next we use \eqref{correct H factorization} to write
\begin{align}
\label{eq2: essence}
\lambda\mathcal{E}_\lambda \cH\P_\lambda f&= -\mathcal{U}_\lambda w^{-1}A\nabla_x\P_\lambda f+\lambda\mathcal{E}_\lambda \dhalf \HT \dhalf \P_\lambda f.
\end{align}
Using Lemma~\ref{le8-}(i) and then Lemma \ref{little1}, we see that
\begin{align}
\label{eq: essence t derivative}
\begin{split}
|||\lambda \mathcal{E}_\lambda \dhalf \HT \dhalf \P_\lambda f|||_{2,\mu}
&=|||\lambda \mathcal{E}_\lambda \dhalf \P_\lambda \dhalf \HT f|||_{2,\mu} \\
&\lesssim |||\lambda \dhalf \P_\lambda \dhalf \HT f|||_{2,\mu} \\
&\lesssim  \|\dhalf f\|_{2,\mu}.
\end{split}
\end{align}
Finally, we bring the principal part approximation into play. We use $\mathcal{U}_\lambda$ and $\mathcal{R}_\lambda$ to write
\begin{align}
\label{eq3: essence}
\begin{split}
\mathcal{U}_\lambda w^{-1}A\nabla_x\P_\lambda f
&= \mathcal{U}_\lambda w^{-1}A \P_\lambda\nabla_x f \\
&=\mathcal{R}_\lambda \P_\lambda \nabla_x f+(\mathcal{U}_\lambda w^{-1} A ) \mathcal{A}_{\lambda} (\P_\lambda -\mathcal{A}_{\lambda}) \nabla_x f  +(\mathcal{U}_\lambda w^{-1} A ) \mathcal{A}_{\lambda} \nabla_x f,
\end{split}
\end{align}
where we have also used that $(\mathcal{A}_\lambda)^2 = \mathcal{A}_\lambda$ for the last term.
Applying  Proposition~\ref{le10} and Lemma~\ref{little1}, we have
\begin{eqnarray*}
|||\mathcal{R}_\lambda \P_\lambda \nabla_x f|||_{2,\mu}
\lesssim |||\lambda \nabla_x  \P_\lambda \nabla_x f|||_{2,\mu}+|||\lambda^{2} \partial_t  \P_\lambda \nabla_x f|||_{2,\mu}
\lesssim \|{\mathbb D} f\|_{2,\mu}.
\end{eqnarray*}
Also, by Lemma \ref{little3} and Lemma \ref{cor5.6}, we have
\begin{align*}
   |||(\mathcal{U}_\lambda w^{-1} A ) \mathcal{A}_{\lambda} (\P_\lambda -\mathcal{A}_{\lambda}) \nabla_x f|||_{2,\mu} \lesssim  ||| (\mathcal{A}_{\lambda}-\P_\lambda) \nabla_x f|||_{2,\mu} \lesssim \| \nabla_x f\|_{2,\mu}.
\end{align*}
Looking back at the successive splittings in \eqref{eq1: essence}, \eqref{eq2: essence}, \eqref{eq3: essence}, we see that the only term that has not been treated in the square function norm is $(\mathcal{U}_\lambda w^{-1} A) \mathcal{A}_\lambda \gradx f$. This proves the claim.
\end{proof}

To conclude the square function estimate for the final term $(\mathcal{U}_\lambda w^{-1} A) \mathcal{A}_\lambda \gradx f$, we establish Lemma \ref{ilem2--}  below. The lemma states that $$|\mathcal{U}_\lambda w^{-1}A|^2\, \frac{\d \mu \d\lambda}{\lambda}$$ is a Carleson measure and that we have good control of the constants. Hence,
\begin{eqnarray*}
|||(\mathcal{U}_\lambda w^{-1}A)\mathcal{A}_{\lambda} \nabla_x f|||_{2,\mu}\lesssim\| \nabla_x f\|_{2,\mu}
\end{eqnarray*}
follows by Carleson's inequality for parabolic cubes, see Lemma \ref{lcarleson}.  This completes the proof of the estimate in \eqref{kee}, and hence the proof of Theorem \ref{thm:Kato} modulo Lemma \ref{ilem2--}. The reader should observe that in our proof of \eqref{kee} we have split off the time derivative $\partial_t$ from $\cH$ and we have controlled the part coming from $\partial_t$ by a standard Littlewood-Paley estimate in \eqref{eq: essence t derivative}

For convenience, we include a proof of the version of Carleson's inequality that is used above. We adapt the elegant dyadic argument found in \cite[Thm. 4.3]{Morris}.
\begin{lem} \label{lcarleson}
Let $\nu$ be a Borel measure on $\ree \times \R^+$ that satisfies
    \begin{align*}
        %\label{carleson's measure}
        \|\nu\|_{\mathcal{C}} := \sup_{\Delta} \frac{\nu(\Delta \times (0,\ell(\Delta)])}{\mu(\Delta)} < \infty,
    \end{align*}
    where the supremum is taken over all dyadic parabolic cubes $\Delta \subset \ree$. Then there is a constant $c$ that only depends on $n$ and $[w]_{A_2}$ such that for every $f \in \L^2_{\mu}$,
\begin{align*}
    %\label{carleson's inequality}
   \int_0^\infty \iint_{\R^{n+1}} |\mathcal{A}_{\lambda}  f(x,t)|^2 \d \nu(x,t,\lambda) \leq c \|\nu\|_{\mathcal{C}} \iint_{\ree} |f|^2 \, \d \mu.
\end{align*}
\end{lem}
\begin{proof}
For $i\in\mathbb Z$ let $\{\Delta_i^j\}_j$ be the partition of $\mathbb R^{n+1}$ into dyadic parabolic cubes such that $\ell(\Delta_i^j)=2^i$. We have
\begin{align*}
    \int_0^\infty \iint_{\R^{n+1}} |\mathcal{A}_{\lambda}  f(x,t)|^2 \d \nu(x,t,\lambda)
    %& \lesssim \sum_{i= -\infty}^{\infty} \sum_{j} \int_{2^{i-1}}^{2^i} \iint_{\Delta^j_i} \biggl|\bariint_{\Delta^j_i} f \, \d y \d s\biggr|^2 \d \nu(x,t,\lambda) \\
    &=  \sum_{i= -\infty}^{\infty} \sum_{j} \biggl|\bariint_{\Delta^j_i} f \, \d y \d s\biggr|^2 \nu(\Delta^j_i \times (2^{i-1},2^i])\\
    &= \sum_{i= -\infty}^{\infty} \sum_{j} |f_i^j|^2 \nu_i^j ,
\end{align*}
where we have introduced $\nu_i^j:= \nu(\Delta^j_i \times (2^{i-1},2^i])$ and $f_i^j:=\bariint_{\Delta^j_i} f \, \d y \d s$. For $r>0$ let $\{\Delta_k(r)\}_k$ be the collection of maximal dyadic parabolic cubes $\Delta_i^j$ such that $|f_i^j| >r$. Note that these cubes are pairwise disjoint and contained in $\{\Max^{(1)} \Max^{(2)} f > r \}$. Hence,
\begin{align*}
   \sum_{i= -\infty}^{\infty} \sum_{j} |f_i^j|^2 \nu_i^j
   	&= \int_0^{\infty} 2r  \sum_{i=-\infty}^{\infty} \sum_{j} 1_{\{|f_i^j| > r\}}  \nu_i^j  \d r \\
 	&\leq \int_{0}^{\infty} 2r \sum_k \sum_{\Delta \subset \Delta_k(r)} \nu(\Delta \times (\ell(\Delta)/2,\ell(\Delta)] ) \, \d r \\
 	&= \int_0^{\infty} 2r \sum_k \nu(\Delta_k(r) \times (0,\ell(\Delta_k(r))]) \, \d r \\
	& \leq \|\nu\|_{\mathcal{C}} \int_0^{\infty} 2r \sum_k \mu(\Delta_k(r)) \, \d r \\
 	& \leq \|\nu\|_{\mathcal{C}} \int_0^{\infty} 2r \mu(\{ \mathcal{M}^{(1)} \mathcal{M}^{(2)}f > r\}) \, \d r \\
 	& = \|\nu\|_{\mathcal{C}} \|\mathcal{M}^{(1)} \mathcal{M}^{(2)}f\|_{2,\mu}^2.
\end{align*}
Now, the claim follows from the Hardy--Littlewood--Muckenhoupt inequality.
\end{proof}

The rest of the section is devoted to the proof of the following lemma.

\begin{lem}
 \label{ilem2--}
For all dyadic parabolic cubes $\Delta=Q\times I\subset\mathbb R^{n+1}$,
\begin{eqnarray*}\label{crucacar+}\int_0^{\ell(\Delta)}\iint_{\Delta}|\mathcal{U}_\lambda w^{-1}A|^2\frac {\d\mu\d\lambda}\lambda\lesssim \mu(\Delta).
\end{eqnarray*}
\end{lem}

As in most related work, the proof of Lemma \ref{ilem2--} is based on the use of appropriate local $Tb$-type test functions.
%%%%%%%%%%%%%%%%%%%%%%%%%%%%%%%%%%%%%%%%%%%%%%%%%%%%%%%%%%%%%%%%%%%%%%%%%%%%
\subsection{Construction of appropriate local Tb-type test functions}\label{subtb}

Let $\zeta\in \IC^{n}$ with $|\zeta|=1$ and let  $\zeta_{i}$ denote the $i$-th component of $\zeta$ for $1\leq i\leq n$.
We let $\chi, \eta$ be smooth functions on $\R^{n}$ and $\R$, respectively, whose values are in $[0,1]$. The function $\chi$ is equal to $1$ on $[-1/2,1/2]^n$ and has support in $(-1,1)^{n}$, and $\eta$ is equal to $1$ on $[-1/4,1/4]$ with support in $(-1,1)$. We fix a parabolic dyadic cube $\Delta$ and denote its center by $(x_{\Delta}, t_{\Delta})$. We first introduce
\begin{align*}
 \chi_{\Delta}(x,t) &:= \chi \bigg( \frac{x-x_{\Delta}}{\ell(\Delta)}\bigg) \eta \bigg(\frac{t-t_{\Delta}}{\ell(\Delta)^2}\bigg).
 \end{align*}
Based on $\zeta$ and $\chi_{\Delta}$, we introduce
\begin{align*}
 L^\zeta_{\Delta}(x,t):= \chi_{\Delta}(x,t)(\Phi_\Delta(x)\cdot \overline{\zeta}),\ \Phi_\Delta(x):=(x-x_\Delta).
\end{align*}
Clearly, $L^\zeta_{\Delta} \in \E_\mu$. Using the function $ L^\zeta_{\Delta}$ and  $0<\epsilon\ll 1$, we define the test function
\begin{align}\label{testfunction}
 f^\zeta_{\Delta,\epsilon}:=\mathcal{E}_{\epsilon\ell(\Delta)} L^\zeta_{\Delta}=(I+(\epsilon\ell(\Delta))^2\cH)^{-1} L^\zeta_{\Delta}.
 \end{align}

\begin{lem}
\label{laa}
Let $\zeta\in \IC^{n}$ with $|\zeta|=1$ and $0<\epsilon\ll 1$ be a degree of freedom. Given a parabolic dyadic cube $\Delta$, define $f^\zeta_{\Delta,\epsilon }$ as in  \eqref{testfunction}. Then,
\begin{align*}
\mathrm{(i)}&\quad \|f^\zeta_{\Delta, \epsilon }-L^\zeta_{\Delta}\|_{2,\mu}^2\, \d\mu\lesssim (\epsilon \ell(\Delta))^2\mu(\Delta),\notag\\
\mathrm{(ii)}&\quad \|{\mathbb D}(f^\zeta_{\Delta,\epsilon}-L^\zeta_{\Delta})\|_{2,\mu}^2\, \d\mu\lesssim \mu(\Delta),\notag\\
\mathrm{(iii)}&\quad \| \mathbb{D} f^\zeta_{\Delta, \epsilon}\|^2_{2,\mu} \lesssim \mu(\Delta).
\end{align*}
\end{lem}
\begin{proof}
Note that
\begin{align*}
f^\zeta_{\Delta,\epsilon }-L^\zeta_{\Delta}
&=-(\epsilon\ell(\Delta))^2\mathcal{E}_{\epsilon\ell(\Delta)}\cH L^\zeta_{\Delta}\notag\\
&=-(\epsilon\ell(\Delta))^2\mathcal{E}_{\epsilon\ell(\Delta)}\dhalf \HT \dhalf L^\zeta_{\Delta}+(\epsilon\ell(\Delta))^2\mathcal{E}_{\epsilon\ell(\Delta)}
w^{-1}\div_xw (w^{-1} A\nabla_xL^\zeta_{\Delta}).\end{align*}
Hence, using the uniform $\L^2_\mu$-boundedness of $(\epsilon\ell(\Delta))\mathcal{E}_{\epsilon\ell(\Delta)}\dhalf$ and $(\epsilon\ell(\Delta))\mathcal{E}_{\epsilon\ell(\Delta)}w^{-1}\div_x w$, see Lemma \ref{le8-}, we get
\begin{align*}
\iint_{\mathbb R^{n+1}}|f^\zeta_{\Delta,\epsilon }-L^\zeta_{\Delta}|^2\, \d\mu&\lesssim \iint_{\mathbb R^{n+1}}|(\epsilon\ell(\Delta)){\mathbb D} L^\zeta_{\Delta}|^2\, \d\mu.
\end{align*}
Furthermore,
\begin{align}
\label{eq1: laa}
\iint_{\mathbb R^{n+1}}|{\mathbb D}L^\zeta_{\Delta}|^2\, \d\mu
&= \iint_{\mathbb R^{n+1}}|\nabla_x L^\zeta_{\Delta}|^2\, \d\mu+\iint_{\mathbb R^{n+1}}|\dhalf L^\zeta_{\Delta}|^2\, \d\mu
\lesssim \mu(\Delta)
\end{align}
by the construction of $L^\zeta_{\Delta}$ (to estimate $\dhalf L^\zeta_{\Delta}$ we use the homogeneity of the Fourier symbol).  Similarly, we deduce that
\begin{eqnarray*}
\iint_{\mathbb R^{n+1}}|{\mathbb D}(f^\zeta_{\Delta,\epsilon}-L^\zeta_{\Delta})|^2\, \d\mu\lesssim \mu(\Delta).
\end{eqnarray*}
This proves $\mathrm{(i)}$ and $\mathrm{(ii)}$. To prove $\mathrm{(iii)}$, we simply use $\mathrm{(ii)}$ and \eqref{eq1: laa}.
\end{proof}

\begin{lem}
\label{ilem2--+}
Given a parabolic dyadic cube $\Delta=Q\times I$, let $f^\zeta_{\Delta,\epsilon }$ be defined as in  \eqref{testfunction}. There exist $\epsilon\in (0,1)$, depending only on the structural constants, and a finite set $W$ of unit vectors in
 $\mathbb C^{n}$, whose
 cardinality depends on $\epsilon$ and $n$, such that
 \begin{eqnarray*}
\sup_{\Delta} \frac 1{|\Delta|}\int_0^{\ell(\Delta)}\iint_{\Delta}|\mathcal{U}_\lambda w^{-1}A|^2\frac {\d\mu\d\lambda}\lambda\lesssim \sum_{\zeta\in W} \sup_{\Delta} \frac 1{|\Delta|}\int_0^{\ell(\Delta)}\iint_{\Delta}|(\mathcal{U}_\lambda w^{-1}A) \mathcal{A}_\lambda \nabla_x f^\zeta_{\Delta,\epsilon }|^2\frac {\d\mu\d\lambda}\lambda,
 \end{eqnarray*}
where the supremum is taken over all dyadic parabolic cubes $\Delta \subset \ree$.
\end{lem}

\begin{proof}
Consider a degree of freedom $\epsilon>0$. Given a unit vector $\zeta$ in $\mathbb C^{n}$, we introduce the cone $$C_\zeta^{\epsilon}:=\{u\in \mathbb C^{n}:\ |u-(u\cdot \bar \zeta)\zeta|\leq \epsilon |u\cdot \bar \zeta|\}.$$
We  note that we can cover $\mathbb C^{n}$  by a finite number of such cones $\{C_\zeta^\epsilon\}$. The number of cones that are needed depends on $\epsilon$ and $n$. In the following, we fix one $C_\zeta^\epsilon$. We let
$$\gamma_{\lambda,\zeta}^{\epsilon}(x,t):=1_{C_\zeta^{\epsilon}}(\mathcal{U}_\lambda w^{-1}A(x,t))\mathcal{U}_\lambda w^{-1}A(x,t)$$ and consider a fixed dyadic parabolic cube $\Delta=Q\times I\subset\mathbb R^{n+1}$.

\subsection*{Step 1: Estimate of the test function along \texorpdfstring{$\boldsymbol{\overline{\zeta}}$}{z}}

We first estimate
\begin{align}
\label{cla+-a}
\iint_{\Delta} (1-\nabla_x f^\zeta_{\Delta,\epsilon }\cdot \zeta) \,  \d x \d t .
\end{align}
To start the estimate, we write
\begin{eqnarray*}
1-\nabla_xf^\zeta_{\Delta,\epsilon }\cdot \zeta=\nabla_xg^\zeta_{\Delta,\epsilon }\cdot \zeta+(1-\nabla_xL^\zeta_{\Delta}\cdot \zeta),
\end{eqnarray*}
where $g^\zeta_{\Delta,\epsilon }:=L^\zeta_{\Delta}-f^\zeta_{\Delta,\epsilon }$. By construction, we have $\nabla_x L^\zeta_{\Delta}(x,t)=\overline{\zeta}$
whenever $(x,t)\in \Delta$. Hence,
\begin{align*}
%\label{cla+-}
\iint_{\Delta}(1-\nabla_x L^\zeta_{\Delta}\cdot  \zeta)\, \d x \d t  =0.
\end{align*}
We have to estimate the contribution to the integral in \eqref{cla+-a} coming from $\nabla_x g^\zeta_{\Delta,\epsilon }\cdot \zeta$. To do this, let $s\in (0,1)$ yet to be chosen, and  let $\varphi\colon \R^{n+1}\to [0,1]$ be a smooth function which is $1$ on $\Delta_s := (1-s)Q \times (1-s^2)I$, supported on $\Delta$, and satisfies $\|\nabla_{x} \varphi\|_\infty \le c(s\ell(\Delta))^{-1}$, $\|\partial_{t}\varphi\|_{\infty}\le c(s\ell(\Delta))^{-2}$ for a dimensional constant $c>0$. Using $\varphi$, we see that
\begin{eqnarray*}
%\label{cla+}
\iint_{\Delta}\nabla_xg^\zeta_{\Delta,\epsilon }\cdot \zeta\, \d x \d t =\iint_{\Delta}(1-\varphi)\nabla_xg^\zeta_{\Delta,\epsilon }\cdot \zeta\, \d x \d t +\iint_{\Delta}\varphi\nabla_xg^\zeta_{\Delta,\epsilon }\cdot \zeta\,\d x \d t  =: \I + \II.
\end{eqnarray*}
Using the Cauchy--Schwarz inequality, Lemma~\ref{laa}$\mathrm{(ii)}$ and \eqref{ainfw} for the $A_2$-weight $\mu^{-1}(x,t) = w^{-1}(x)$, we obtain
\begin{equation*}
\begin{aligned}
|\textup{I}|& \leq
\biggl(\iint_{\Delta} |1-\varphi|^2 \, \d \mu^{-1} \biggr)^{\frac{1}{2}} \biggl(\iint_{\Delta} |\nabla_x g^\zeta_{\Delta,\epsilon }|^2 \, \d \mu \biggr)^{\frac{1}{2}}
\\ &\lesssim
\mu^{-1}(\Delta \setminus \Delta_s)^{\frac{1}{2}}\mu(\Delta)^{\frac{1}{2}}\\
&\lesssim s^\eta \mu^{-1}(\Delta)^{\frac{1}{2}} \mu(\Delta)^{\frac{1}{2}} \\
&\leq s^\eta [w]_{A_2} |\Delta|.
\end{aligned}
\end{equation*}
To estimate $\II$, we integrate by parts
\begin{align*}
 \II = -\iint_{\R^{n+1}} g^\zeta_{\Delta,\epsilon } \nabla_x\varphi\cdot \zeta\, \d x \d t
\end{align*}
and using the Cauchy--Schwarz inequality and Lemma \ref{laa}$\mathrm{(i)}$, we obtain similarly
\begin{align*}
| \textup{II}|
& \leq \biggl(\iint_{\ree} |\nabla_x\varphi|^2 \, \d \mu^{-1} \biggr)^{\frac{1}{2}} \biggl(\iint_{\ree} |g^\zeta_{\Delta,\epsilon }|^2 \, \d \mu \biggr)^{\frac{1}{2}}  \\
& \lesssim (s\ell(\Delta))^{-1}\mu(\Delta)^{\frac{1}{2}} \epsilon \ell(\Delta) \mu^{-1}(\Delta)^{\frac{1}{2}} \\
& \leq \eps s^{-1} [w]_{A_2} |\Delta|.
\end{align*}
We now choose $s=\epsilon^{1/(\eta+1)}$, so that the estimates for $\I$ and $\II$ come with the same power of $\eps$. Putting the estimates together, we obtain for the integral in \eqref{cla+-a} that
\begin{eqnarray}\label{est1}
\frac 1 {|\Delta|}\biggl |\iint_{\Delta} 1-\nabla_xf^\zeta_{\Delta,\epsilon }\cdot \zeta\, \d x \d t \biggr |\lesssim \epsilon^{\frac{\eta}{\eta+1}}.
\end{eqnarray}
Using Lemma \ref{laa}$\mathrm{(iii)}$  and the Cauchy--Schwarz inequality, we also see that
\begin{equation}
\begin{aligned}
\label{est3/2}
\frac 1 {|\Delta|}\iint_{\Delta}| \nabla_x f^\zeta_{\Delta,\epsilon }|\, \d x \d t \leq \frac{1}{|\Delta|} \biggl(\iint_{\Delta}| \nabla_x f^\zeta_{\Delta,\epsilon }|^2 \, \d \mu \biggr)^{\frac{1}{2}} \mu^{-1}(\Delta)^{\frac{1}{2}}  \lesssim 1.
\end{aligned}
\end{equation}

\subsection*{Step 2: Choice of $\eps$}

Using the estimates in the last two displays, we see, if $\epsilon$ is chosen small enough, that
\begin{eqnarray*}\label{est2}
\frac 1 {|\Delta|}\iint_{\Delta}\mbox{Re}( \nabla_x f^\zeta_{\Delta,\epsilon }\cdot \zeta)\, \d x \d t \geq \frac 7 8,
\end{eqnarray*}
and
\begin{eqnarray*}\label{est3}
\frac 1 {|\Delta| }\iint_{\Delta}| \nabla_x f^\zeta_{\Delta, \epsilon }| \, \d x \d t \leq c,
\end{eqnarray*}
for some large constant $c$ depending only on the structural constants.
We now perform a stopping time decomposition as in \cite{AHLMcT} to select a collection $\mathcal{S}_\zeta' = \{\Delta'\}$ of dyadic parabolic subcubes of $\Delta$, which are maximal with respect
to the property that either
\begin{eqnarray}\label{est4}
\frac 1 {|\Delta'|}\iint_{\Delta'}\mbox{Re}( \nabla_x f^\zeta_{\Delta,\epsilon }\cdot \zeta)\, \d x \d t \leq \frac 3 4,
\end{eqnarray}
or
\begin{eqnarray}\label{est5}
\frac 1 {|\Delta'| }\iint_{\Delta'}| \nabla_x f^\zeta_{\Delta, \epsilon }| \, \d x \d t \geq (4\epsilon)^{-2},
\end{eqnarray}
holds. In other words, we parabolically dyadically subdivide  $\Delta$ and stop the first  time either \eqref{est4} or \eqref{est5} hold. Then, $S'_\zeta=\{\Delta'\}$ is a disjoint set of the parabolic dyadic subcubes of $\Delta$. Let
$\mathcal{S}_\zeta''=\{\Delta''\}$ be the collection of all the parabolic dyadic subcubes of
$\Delta$ not contained in any $\Delta'\in S'_\zeta$. Then, each $\Delta''\in \mathcal{S}_\zeta''$ satisfies
\begin{align}\label{est4a}
\begin{split}
	(a)&\quad \frac 1 {|\Delta''| }\iint_{\Delta''}\mbox{Re}( \nabla_x f^\zeta_{\Delta,\epsilon }\cdot \zeta)\, \d x \d t \geq \frac 3 4,\\
	(b)&\quad \frac 1 {|\Delta''| }\iint_{\Delta''}| \nabla_x f^\zeta_{\Delta, \epsilon }| \, \d x \d t \leq (4\epsilon)^{-2}.
\end{split}
\end{align}
At this stage, we claim that by the same type of argument as in the proof of statement $(i)$ in Proposition 5.7 in \cite{AHLMcT} there exists $\epsilon \in (0,1)$ small and depending only on the structural constants, and $\eta'=\eta'(\epsilon)\in (0,1)$ such that
\begin{align}\label{est4ah0}
\bigg|\bigcup_{\Delta' \in \mathcal{S}_\zeta'} \Delta'\bigg| \leq (1-\eta')|\Delta|.
\end{align}
In particular, from now on $\epsilon$ is fixed.  For the convenience of the reader, we include a proof here.

Let $E_1$ and $E_2$ be the unions of all parabolic cubes in $\mathcal{S}_\zeta'$ which satisfy \eqref{est4} and \eqref{est5}, respectively. Then, $$\bigg|\bigcup_{\Delta' \in \mathcal{S}_\zeta'} \Delta' \bigg|\leq |E_1| + |E_2|.$$
For $|E_2|$, we have
$$
|E_2| \leq (4 \epsilon)^2 \sum_{\Delta' \in \mathcal{S}_\zeta'} \iint_{\Delta'} | \nabla_x f^\zeta_{\Delta,\epsilon }|\, \d x \d t \leq (4 \epsilon)^2 \iint_{\Delta} | \nabla_x f^\zeta_{\Delta,\epsilon }|\, \d x \d t \leq  (4 \epsilon)^2 c |\Delta|,
$$
where we used \eqref{est3/2} in the last step. To control $|E_1|$, we let $h := 1-\Re (\nabla_x f^\zeta_{\Delta,\epsilon }\cdot \zeta)$ and write
\begin{align}
\label{bolle}
    |E_1| \leq 4 \sum_{\Delta'} \iint_{\Delta'} h \, \d x \d t = 4  \iint_{\Delta} h \, \d x \d t -  4  \iint_{\Delta \setminus E_1} h \, \d x \d t,
\end{align}
where the sum is taken over all parabolic subcubes of $E_1$. By \eqref{est1}, the first term on the right is controlled by $\epsilon^{\eta/(\eta+1)} |\Delta|$ times a constant depending on the structure constants. Using in succession the Cauchy--Schwarz inequality, Lemma \ref{laa}$\mathrm{(iii)}$, the $A_2$-property and Young's inequality, the second term on the right is controlled by
\begin{align*}
    &4|\Delta \setminus E_1| + 4 \mu^{-1}(\Delta \setminus E_1)^{\frac 1 2} \biggl(\iint_{\Delta} |\nabla_x f^\zeta_{\Delta,\epsilon }|^2 \, \d \mu\biggr)^{\frac 1 2} \\ &\leq   4|\Delta \setminus E_1| + 4 \widetilde{c} \mu^{-1}(\Delta \setminus E_1)^{\frac 1 2} \mu(\Delta)^{\frac 1 2} \\
    & \leq 4|\Delta \setminus E_1| + 4 \widetilde{c} |\Delta \setminus E_1|^\eta |\Delta|^{1-\eta} \\
    & \leq (4+ \widetilde{c} \eps^{-\frac{1}{\eta}}) |\Delta \setminus E_1|  + \widetilde{c} \eps^{1-\eta} |\Delta|,
\end{align*}
where $\widetilde{c}$ depends on the structural constants and changes from line to line. Going back to \eqref{bolle} and re-arranging terms, we find
\begin{align*}
 |E_1| \leq \frac{4 + \widetilde{c} \eps^{-\frac{1}{\eta}} + \widetilde{c}(\eps^{\frac{\eta}{\eta + 1}} + \eps^{1-\eta})}{5 + \widetilde{c} \eps^{-\frac{1}{\eta}}}|\Delta|,
 %\leq ((\tilde c_1 \epsilon^{\eta/(\eta+1)}+4+\tilde c_3 \epsilon^{2/(2-\eta)}+ \tilde c_4 \epsilon^{-2/\eta })/ (5+\tilde c_4 \epsilon^{-2/\eta }))|\Delta|.
\end{align*}
and taking $\epsilon$ small enough, we conclude \eqref{est4ah0}.

Since $\mu$ is an $A_2$-weight, we obtain from \eqref{est4ah0} and upon taking $\eta'$ smaller in dependence of the structural constants and $\eps$ that \
\begin{align}\label{est4ah}
\mu\bigg (\bigcup_{\Delta' \in \mathcal{S}_\zeta'} \Delta'\bigg) \leq (1-\eta') \mu(\Delta),
\end{align}
see for example \cite[p.~196]{Stein}.
\subsection*{Step 3: Re-introducing the averaging operator}

Given $\Delta$, we consider $\Delta''\in \mathcal{S}_\zeta^{''}$ as above. Set
\begin{eqnarray}\label{est8}
v:=\frac 1{\mu(\Delta'')}\iint_{\Delta''} \nabla_x f^\zeta_{\Delta,\epsilon }\, \d x \d t \in \mathbb C^{n}.
\end{eqnarray}
If $(x,t)\in \Delta''$ and $\ell(\Delta'')/2 < \lambda\leq \ell(\Delta'')$, then $v=(\mathcal{A}_\lambda \nabla_x f^\zeta_{\Delta,\epsilon })(x,t)$. Assume that
$u:=(\mathcal{U}_\lambda w^{-1}A)(x,t) \in C_\zeta^\epsilon$. The pair of vectors $(u,v)$  satisfies the estimates in \eqref{est4a}. Thus, we can apply \cite[Lem.~5.10]{AHLMcT} with $w := \zeta$ and conclude that $|u| \leq 4 |u \cdot v|$, that is,
\begin{eqnarray}\label{est9}
|\gamma_{\lambda,\zeta}^{\epsilon}(x,t)|\leq 4|(\mathcal{U}_\lambda w^{-1}A(x,t))\cdot (\mathcal{A}_\lambda \nabla_x f^\zeta_{\Delta,\epsilon })(x,t)|.
\end{eqnarray}
We next observe that by construction the Carleson box $\Delta\times (0,\ell(\Delta)]$ can be partitioned into Carleson boxes
$\Delta'\times (0,\ell(\Delta')]$, with $\Delta'\in \mathcal{S}_\zeta'$,  and  Whitney boxes $\Delta''\times (\ell(\Delta'')/2,\ell(\Delta'')]$, with $\Delta''\in \mathcal{S}_\zeta''$. In particular,
\begin{eqnarray*}
\frac{1}{\mu(\Delta)} \int_0^{\ell(\Delta)}\iint_{\Delta}|\gamma_{\lambda,\zeta}^{\epsilon}(x,t)|^2\, \frac {\d\mu\d\lambda}\lambda=: \I + \II,
\end{eqnarray*}
where
\begin{align*}
\I&:=\frac{1}{\mu(\Delta)} \sum_{\Delta'\in \mathcal{S}_\zeta'}
\int_0^{\ell(\Delta')}\iint_{\Delta'}|\gamma_{\lambda,\zeta}^{\epsilon}(x,t)|^2 \, \frac {\d\mu\d\lambda}\lambda,\notag\\
\II&:=\frac{1}{\mu(\Delta)} \sum_{\Delta''\in \mathcal{S}_\zeta''}
\int_{\ell(\Delta'')/2}^{l(\Delta'')}\iint_{\Delta''}|\gamma_{\lambda,\zeta}^{\epsilon}(x,t)|^2\, \frac {\d\mu\d\lambda}\lambda.
\end{align*}
Using \eqref{est4ah}, we obtain
$$\I \leq \frac{1}{\mu(\Delta)} \sum_{\Delta'\in \mathcal{S}_\zeta'} A_{\zeta}^{\epsilon}\mu(\Delta')\leq (1-\eta')A_{\zeta}^{\epsilon},$$
where $$A_{\zeta}^{\epsilon}:=\sup_{\widetilde{\Delta}}\frac 1 {\mu(\widetilde{\Delta})}\int_0^{\ell(\widetilde{\Delta})}\iint_{\widetilde{\Delta}}|\gamma_{\lambda,\zeta}^{\epsilon}(x,t)|^2\frac {\d\mu\d\lambda}\lambda,$$
and where the supremum is taken over all dyadic parabolic subcubes $\widetilde{\Delta} \subset \Delta$.  By \eqref{est9}, we have
\begin{align*}
\II\leq\frac{16}{\mu(\Delta)} \int_{0}^{\ell(\Delta)}\iint_{\Delta}|(\mathcal{U}_\lambda w^{-1}A)(x,t) \cdot (\mathcal{A}_\lambda \nabla_x f^\zeta_{\Delta,\epsilon })(x,t)|^2\frac {\d\mu\d\lambda}\lambda.
\end{align*}
Since these estimates hold for all dyadic parabolic cubes, in particular those, who are subcubes of $\Delta$, we conclude that
\begin{align*}
A_{\zeta}^{\epsilon}\leq (1-\eta')A_{\zeta}^{\epsilon} + \sup_{\widetilde{\Delta}} {\frac{16}{\mu(\widetilde{\Delta})}} \int_{0}^{\ell(\widetilde{\Delta})}\iint_{\widetilde{\Delta}}|(\mathcal{U}_\lambda w^{-1}A)(x,t )\cdot (\mathcal{A}_\lambda \nabla_x f^\zeta_{\widetilde{\Delta},\epsilon })(x,t)|^2
\frac {\d\mu\d\lambda}\lambda.
\end{align*}
%where the supremum now is taken over all dyadic parabolic cubes $\tilde{\Delta} \subset \ree$.
Summing with respect to $\zeta\in W$ completes the proof of Lemma~\ref{ilem2--+} under the \emph{a priori} assumption that $A_{\zeta}^{\epsilon}$ is qualitatively finite, since it can then be absorbed into the left-hand side.

\subsection*{Step 4: Removing the a priori assumption} The \emph{a priori} assumption that $A_{\zeta}^{\epsilon}$ is qualitatively finite can be removed by setting $\gamma_{\lambda,\zeta}^{\epsilon}(x,t)$ to $0$ for $\lambda$ small and large, repeating the argument from \eqref{est9} on and passing to the limit at the end. For the truncated $\gamma_{\lambda,\zeta}^{\epsilon}(x,t)$, we get $A_{\zeta}^{\epsilon} < \infty$ from \eqref{ineqwell}. Indeed, for $0<\delta <1$ small we have
\begin{align*}
 A_{\zeta}^{\epsilon} \leq   \int_{\delta}^{\delta^{-1}} \iint_{\Delta}  |\mathcal{U}_{\lambda} w^{-1} A|^2  \frac{\d\mu\d\lambda}{\lambda} \lesssim \mu(\Delta) |\ln(\delta)|.
\end{align*}
This completes the argument.
\end{proof}

\subsection{The Carleson measure estimate: proof of Lemma \ref{ilem2--}} \label{Carleson} Thanks to Lemma \ref{ilem2--+} it suffices to prove
\begin{eqnarray}
\label{ff1}
\int_0^{\ell(\Delta)}\iint_\Delta|(\mathcal{U}_\lambda w^{-1}A) \mathcal{A}_\lambda \nabla_x f^\zeta_{\Delta,\epsilon}|^2\frac {\d\mu\d\lambda}\lambda\lesssim \mu(\Delta).
\end{eqnarray}
The left-hand side in \eqref{ff1} is bounded by
\begin{align*}
  |||(\lambda\mathcal{E}_\lambda \cH + (\mathcal{U}_\lambda w^{-1} A) \mathcal{A}_\lambda \nabla_x)f^\zeta_{\Delta,\epsilon}|||_{2,\mu}^2 + \int_0^{\ell(\Delta)}\iint_\Delta|\lambda \cE_\lambda \cH f^\zeta_{\Delta,\epsilon} |^2\frac {\d\mu\d\lambda}\lambda =: \I + \II.
\end{align*}
By Proposition~\ref{prop: upshot 5 and 7} and Lemma~\ref{laa} we have
\begin{align*}
	\I \lesssim \| \mathbb{D} f^\zeta_{\Delta,\epsilon}\|_{2,\mu}^2 \lesssim \mu(\Delta).
\end{align*}
As for $\II$, we obtain from \eqref{testfunction} that $\cH f^\zeta_{\Delta,\epsilon }=( L^\zeta_{\Delta}-f^\zeta_{\Delta,\epsilon })/(\epsilon \ell(\Delta))^2$. Using the $\L^2_\mu$-boundedness of $\mathcal{E}_\lambda$, see Lemma \ref{le8-}, and then Lemma~\ref{laa}, we obtain
\begin{align*}
	\II
	\lesssim \int_0^{\ell(\Delta)} \|\lambda (\epsilon \ell(\Delta))^{-2}(L^\zeta_{\Delta}-f^\zeta_{\Delta,\epsilon })\|_{2,\mu}^2 \, \frac{\d \lambda}{\lambda}
	= \frac{1}{2 \eps^4 \ell(\Delta)^2 } \|L^\zeta_{\Delta}-f^\zeta_{\Delta,\epsilon }\|_{2,\mu}^2 \lesssim \eps^{-2} \mu(\Delta).
\end{align*}
This completes the proof of \eqref{ff1}, and hence the proof of Theorem~\ref{thm:Kato}.

\def\cprime{$'$} \def\cprime{$'$} \def\cprime{$'$}

\end{document}